\documentclass[12pt,reqno] {amsart}%
\usepackage{amsfonts}
\usepackage{amsmath}
\usepackage{amssymb}
\usepackage{graphicx}%
\usepackage{hyperref}
\hypersetup{
    colorlinks = true,
    linkcolor = blue,
    urlcolor = blue,
    citecolor = blue,
    anchorcolor = black,
}

\usepackage{cleveref}
\usepackage{comment}
 \usepackage{mathrsfs}
 \usepackage{float}
 \usepackage{tikz}
 \usepackage{bm}
 \usepackage{ulem}
\usepackage{slashed} 
 \usepackage{braket} 
\usepackage{enumerate}
\usepackage{rotating}

\allowdisplaybreaks

\makeatletter
\def\l@section{\@tocline{1}{0pt}{1pc}{}{}}
\def\l@subsection{\@tocline{2}{0pt}{1pc}{4.6em}{}}
\def\l@subsubsection{\@tocline{3}{0pt}{1pc}{7.6em}{}}
\renewcommand{\tocsection}[3]{%
  \indentlabel{\@ifnotempty{#2}{\makebox[2.3em][l]{%
    \ignorespaces#1 #2.\hfill}}}#3}
\renewcommand{\tocsubsection}[3]{%
  \indentlabel{\@ifnotempty{#2}{\hspace*{2.3em}\makebox[2.3em][l]{%
    \ignorespaces#1 #2.\hfill}}}#3}
\renewcommand{\tocsubsubsection}[3]{%
  \indentlabel{\@ifnotempty{#2}{\hspace*{4.6em}\makebox[3em][l]{%
    \ignorespaces#1 #2.\hfill}}}#3}
\makeatother
\theoremstyle{plain}
\numberwithin{equation}{section}
\newtheorem{theorem}{Theorem}[section]
\newtheorem{corollary}[theorem]{Corollary}
\newtheorem{lemma}[theorem]{Lemma}
\newtheorem{proposition}[theorem]{Proposition}
\newtheorem{notation}[theorem]{Notation}

\newtheorem{definition}[theorem]{Definition}
\newtheorem{remark}[theorem]{Remark}

\newtheorem{main theorem}{Main Theorem}

\newcommand{\id}{\text{id}}

\newcommand{\be}{\begin{equation}}
\newcommand{\ee}{\end{equation}}

\renewcommand{\leq}{\leqslant}
\renewcommand{\geq}{\geqslant}

\newcommand{\CP}{\mathbf{CP}_{\mathcal{N}}(\mathcal{M})}
\newcommand{\CB}{\mathbf{CB}_{\mathcal{N}}(\mathcal{M})}
\begin{document}

\title{Phase Group Categories of Bimodule Quantum Channels}

\author{Linzhe Huang}
\address{Linzhe Huang, Beijing Institute of Mathematical Sciences and Applications, Beijing, 101408, China}
\email{huanglinzhe@bimsa.cn}


\author{Chunlan Jiang}
\address{Chunlan Jiang, School of Mathematical Sciences, Hebei Normal University, Shijiazhuang, Hebei, 050024, China}
\email{cljiang@hebtu.edu.cn}

\author{Zhengwei Liu}
\address{Zhengwei Liu, Yau Mathematical Sciences Center and Department of Mathematics, Tsinghua University, Beijing, 100084, China \\
 Beijing Institute of Mathematical Sciences and Applications, Beijing, 101408, China}
\email{liuzhengwei@mail.tsinghua.edu.cn}

\author{Jinsong Wu}
\address{Jinsong Wu, Beijing Institute of Mathematical Sciences and Applications, Beijing, 101408, China}
\email{wjs@bimsa.cn}
\date{}

\maketitle

\begin{abstract}
In this paper, we study the quantum channel on a von Neuamnn algebra $\mathcal{M}$ preserving a von Neumann subalgebra $\mathcal{N}$, namely an $\mathcal{N}$-$\mathcal{N}$-bimodule unital completely positive map.
By introducing the relative irreducibility of a bimodule quantum channel, 
we show that its eigenvalues with modulus 1 form a finite cyclic group, called its phase group. Moreover, the corresponding eigenspaces are invertible $\mathcal{N}$-$\mathcal{N}$-bimodules, which encode a categorification of the phase group.
When $\mathcal{N}\subset \mathcal{M}$ is a finite-index irreducible subfactor of type II$_1$, we prove that any bimodule quantum channel is relatively irreducible for the intermediate subfactor of its fixed points. In addition, we can reformulate and prove these results intrinsically in subfactor planar algebras without referring to the subfactor using the methods of quantum Fourier analysis.
\end{abstract}



\section{Introduction}
Quantum channels are widely studied in quantum information.
In finite-dimensional systems, a quantum channel is a trace-preserving completely positive map acting on matrix algebras.
In infinite-dimensional  systems,  a von Neumann algebra $\mathcal{M}$ is generated by  observables, which may not have trace. 
A quantum channel is a unital completely positive map on $\mathcal{M}$ in the Heisenberg picture.
Furthermore, we study  quantum channels with symmetry $\mathcal{N}$ as   $\mathcal{N}$-$\mathcal{N}$-bimodule  maps for a von Neumann subalgebra $\mathcal{N}$.
There are various motivations to study $\mathcal{N}$-$\mathcal{N}$-bimodule  maps.
One can consider $\mathcal{N}$ as conserved quantities preserved by quantum channels; 
such a quantum channel can be reformulated as a unital completely positive $\mathcal{N}$-$\mathcal{N}$-bimodule  map by Kadison-Schwarz inequality.
 In quantum Markov processes \cite{DKY19,Gud08,SPF14}, we can consider $\mathcal{N}$ as  observables in the past invariant under the quantum channel at the current time.
 In subfactor theory, when $\mathcal{N}\subseteq\mathcal{M}$ is a subfactor, $\mathcal{N}$-$\mathcal{N}$ bimodules are well studied.
In quantum field theory, one can consider $\mathcal{N}$ as the observables in a local region (see \cite{LK95}).

In classical Markov progresses, Perron-Frobenius theorem is very fundamental.
Evans and H{\o}egh-Krohn \cite{EvaHoe78} generalized this theorem to irreducible quantum channels on finite-dimensional von Neumann algebras, which is called quantum Perron-Frobenius theorem.



In this paper, we study bimodule quantum channels with quantum symmetry $\mathcal{N}$ and establish the Perron-Frobenius theorem with quantum symmetries, which could be infinite-dimensional.
\begin{theorem}[See Theorem \ref{thm:RI Frobenius factor} and Theorem \ref{thm:RI Frobenius finite vN algebra}]\label{thm:phase group}
Suppose that $\mathcal{N}\subseteq\mathcal{M}$ is finite-index and $\Phi$ is a relatively irreducible bimodule quantum channel.
Suppose $\mathcal{N}$ is a factor.
Then the eigenvalues of $\Phi$ with modulus $1$ form a finite cyclic group, which we call the phase group.
\end{theorem}
Here, our notation of relative irreducibility for $\Phi$ means that $\Phi(p)\leq \lambda p$ implies $p\in\mathcal{N}$ for any projection $p\in\mathcal{M}$.
The relative irreducibility reduces to the irreducibility introduced by  Evans and H{\o}egh-Krohn \cite{EvaHoe78} when $\mathcal{N}=\mathbb{C}\mathbf{1}$.
The finite-index condition means  that the inclusion  $\mathcal{N}\subseteq\mathcal{M}$ is finite-index if there exist a conditional expectation $E_{\mathcal{N}}$ and $\lambda>0$ such that the Pimsner-Popa inequality holds \cite{PimPop86}:
\begin{align*}
    E_{\mathcal{N}}(x)\geq \lambda x,\quad\text{ for all } x\in\mathcal{M}^+.
\end{align*}
The best constant $\lambda^{-1}$ for the minimal conditional expectation is the Jones index.

When $\mathcal{N}=\mathbb{C}\mathbf{1}$ and $\mathcal{M}$ is finite-dimensional, Theorem \ref{thm:phase group} reduces to the quantum Frobenius theorem of Evans and H{\o}egh-Krohn.
They also proved that the eigenspace at each phase is one-dimensional and it is generated by a unitary.
In our case, the eigenspace of $\Phi$ at each phase is an $\mathcal{N}$-$\mathcal{N}$-bimodule, which may be an infinite-dimensional Hilbert space.
\begin{theorem}[See Theorem \ref{thm:RI Frobenius factor} and Theorem \ref{thm:RI Frobenius finite vN algebra}]\label{thm:bimodule fusion category}
Suppose $\Phi$ is relatively irreducible and $\mathcal{N}$ is a factor.
The eigenspaces of $\Phi$  are $\mathcal{N}$-$\mathcal{N}$-bimodules $\{H_{\alpha}\}_{\alpha\in \Gamma}$ and they form a unitary fusion category, which is the categorification of the phase group $\Gamma$.
In particular, the quantum dimension of each bimodule $H_{\alpha}$ is 1.
\end{theorem}
For each $\mathcal{N}$-$\mathcal{N}$-bimodule $H_{\alpha}$, there exists a unitary $u_{\alpha}\in\mathcal{M}$ such that $H_{\alpha}=u_{\alpha}\mathcal{N}=\mathcal{N}u_{\alpha}$.


Next we consider $\mathcal{N}\subseteq \mathcal{M}$ to be an irreducible finite-index subfactor of type II$_1$.
Note that if $\mathcal{M}$ is finite-dimensional, then  $\mathcal{N}=\mathcal{M}$ and $\Phi$ is the identity, which is trivial.
However, the theory is non-trivial in the infinite-dimensional case. 
We prove that $\Phi$ is a relatively irreducible $\mathcal{P}$-$\mathcal{P}$-bimodule map, where $\mathcal{P}$ is the intermediate subfactor given by the fixed points.
So the eigenvalues of $\Phi$ with modulus $1$ form the phase group and the eigenspaces are $\mathcal{P}$-$\mathcal{P}$-bimodules, which form a unitary categorification of the phase group.
The quantum dimension of each eigenspace as $\mathcal{N}$-$\mathcal{N}$-bimodule is the Jones index $[\mathcal{P}:\mathcal{N}]$.

We can use subfactor planar algebras to characterize a bimodule map $\Phi$ as an element $y_{\Phi}$ in 2-box spaces.
The spectrum of $\Phi$ is the same as the spectrum of $y_{\Phi}$.
The bimodule map $\Phi$ is completely positive if and only if $y_{\Phi}$ is $\mathfrak{F}$-positive, i.e., its Fourier transform is positive.
The projection from $L^2(\mathcal{M})$ onto $H_{\alpha}$ is the spectral projection of $y_{\Phi}$ with the spectral point $\alpha$.
In this case, we can give a canonical construction of the unitary $u_{\alpha}$ such that the 3-cocycle $\omega$ is constant 1.
We obtain a sequence of factors $\displaystyle \mathcal{N}\subseteq\mathcal{P}\subseteq \bigoplus_{\alpha\in\Gamma} H_{\alpha}\subseteq\mathcal{M}$ and $\displaystyle \mathcal{P}\subseteq \bigoplus_{\alpha\in\Gamma} H_{\alpha}$ is the group subfactor of the phase group.
The biprojections $P$ and $Q$ corresponding to the intermediate subfactors $\mathcal{P}$ and $\displaystyle \bigoplus_{\alpha\in\Gamma} H_{\alpha}$ are generated by the Fourier transforms of  $\Phi\Phi^*$ and $\Phi$ respectively.
We reformulate this statement intrinsically in subfactor planar algebras without referring to the subfactor $\mathcal{N}\subseteq\mathcal{M}$ as follows.
\begin{theorem}[See Theorem \ref{thm:Frobenius irreducible} and Theorem \ref{thm:two biprojection}]\label{thm:intrinsically}
    Suppose $\mathscr{P}$ is an irreducible subfactor planar algebra.
    Let $x\in\mathscr{P}_{2,\pm}$ be a positive operator.
    Let $q$ be the biprojection generated by $x$, and let $p$ be the biprojection generated by $x\ast\overline{x}$.
  Then $\displaystyle q=\sum_{\alpha\in\Gamma} p_{\alpha}$, where $p_{\alpha}$ is a right shift of the biprojection $p$.
  Moreover $\{p_{\alpha}\}_{\alpha\in\Gamma}$ forms a finite cyclic group under convolution.
\end{theorem}
Theorem \ref{thm:intrinsically} implies Theorem \ref{thm:phase group} and Theorem \ref{thm:bimodule fusion category} in the case of irreducible type II$_1$ subfactors.
We can give an intrinsic  proof on this theorem using methods developed in quantum Fourier analysis \cite{JJLRW20}.


The paper is organized as follows: In \S \ref{sec:Completely positive bimodule map}, we recall the bimodule quantum channels and the Fourier multipliers, and reinterpret the Pimsner-Popa inequality for bimodule quantum channels.
In \S \ref{sec:PF theory for CPB}, we introduce the notion of relative irreducibility for bimodule quantum channels and characterize the eigenspaces of bimodule quantum channels with respect to eigenvalues with modulus 1. 
In \S \ref{sec:Phase groups for F positive elements}, we present an intrinsic proof of the characterization of eigenspaces of bimodule quantum channels in terms of subfactor planar algebras by quantum Fourier analysis.
In \S \ref{sec:extension}, we show the results in \S \ref{sec:PF theory for CPB} for
finite inclusions of finite von Neumann algebras.
We also obtain equivalent characterizations of the relative irreducibility.

\section{Bimodule Quantum Channels}\label{sec:Completely positive bimodule map}
A von Neumann algebra $\mathcal{M}$ on a Hilbert space is a weak-operator closed $*$-algebra.
Let $\mathcal{N}$ be a von Neumann subalgebra of $\mathcal{M}$.
We use $\mathcal{N}\subseteq\mathcal{M}$ to denote the inclusion of von Neumann algebras.
We call $\mathcal{M}$ a factor if it has a trivial center.
A factor $\mathcal{M}$ is called a finite factor if it has a unique normal faithful trace.
A subfactor is an inclusion of factors $\mathcal{N}\subseteq\mathcal{M}$.
Subfactors are quantum symmetries, generalizing the symmetries  of groups and their duals.
The indices of subfactors generalize the indices of subgroups in groups.
Jones \cite{Jon83} proved the remarkable index classification theorem of subfactors:
\[ \left\{ 4 \cos^2\frac{\pi}{n}, n=3,4,5,\ldots \right\} \cup [4,\infty]. \]
Suppose $\mathcal{M}$ and $\mathcal{N}$ are finite factors and $\mathcal{N}\subseteq\mathcal{M}$ is a subfactor of finite index   $[\mathcal{M}:\mathcal{N}]=\mu$, $\mu>0$.
The subfactor $\mathcal{N}\subseteq\mathcal{M}$ is called extremal if the normalized traces on the commutant $\mathcal{N}'$ of $\mathcal{N}$ and $\mathcal{M}$ coincide on $\mathcal{N}'\cap\mathcal{M}$.
We always assume that  $\mathcal{M}$ and $\mathcal{N}$ are finite factors, and all subfactors in this paper are extremal and finite indices if there is no confusion.

Let $\tau_{\mathcal{M}}(=\tau)$  be the normal faithful tracial state on  $\mathcal{M}$.
We denote by $L^2(\mathcal{M})$ the Gelfand-Naimark-Segal (GNS) representation arising from $\tau$ and $\Omega$ the cyclic separating vector.
Let $e_{1}$ be the Jones projection from $L^2(\mathcal{M})$ onto the Hilbert subspace $L^2(\mathcal{N})$.
We denote by $E_{\mathcal{N}}$ the trace-preserving conditional expectation of $\mathcal{M}$ onto $\mathcal{N}$.
Then we have that $e_1xe_1=E_{\mathcal{N}}(x)e_1$ for all $x\in\mathcal{M}$.
Let $\mathcal{M}_1=\langle \mathcal{M}, e_1\rangle$ be the von Neumann algebra acting on $L^2(\mathcal{M})$ generated by $\mathcal{M}$ and $e_1$.
Note that $\mathcal{M}_1=J\mathcal{N}'J$, where $J$ is the conjugation on $L^2(\mathcal{M})$.
We denote by $\Omega$ the cyclic and separating vector in $L^2(\mathcal{M})$ for $\mathcal{M}$.
This is called the basic construction \cite{Jon83} for the subfactor $\mathcal{N}\subseteq\mathcal{M}$.
Let $\tau_1$ be a fixed normal faithful tracial state on $\mathcal{M}_1$.
Continuing the basic construction for the subfactor $\mathcal{M}\subseteq\mathcal{M}_1$, we can obtain the Jones projection $e_2$, the  trace-preserving conditional expectation $E_{\mathcal{M}}$, the cyclic and separating vector $\Omega_2$ and the von Neumann algebra $\mathcal{M}_2$.
We have that $\tau_1|_{\mathcal{M}}=\tau$ and $E_{\mathcal{M}}(e_1)=\mu^{-1} \mathbf{1}$.

Recall that for a subfactor $\mathcal{N}\subseteq\mathcal{M}$, there exists a Pimsner-Popa basis \cite{PimPop86} $\displaystyle \{\eta_j\}_{j=1}^m$ of $\mathcal{N}\subseteq \mathcal{M}$, namely
\begin{align*}
x=\sum_{j=1}^m  E_{\mathcal{N}}(x\eta_j^*)\eta_j, \quad \text{for all } x\in\mathcal{M}.
\end{align*}
Equivalently, $\displaystyle x=\sum_{j=1}^m \eta_j^* E_{\mathcal{N}}(\eta_j x)$.
The index of $\mathcal{N}\subseteq\mathcal{M}$ is $\tau_1(e_1)^{-1}= \displaystyle \sum_{j=1}^m\tau(\eta_j\eta_j^*)=\mu$.
The Pimsner-Popa inequality reads
\begin{align}\label{eq:PP inequality}
    E_{\mathcal{N}}(x)\geq \mu^{-1} x,\quad\text{ for all } x\in\mathcal{M}^+,
\end{align}
where the inverse $\mu^{-1}$ of the Jones index is called the Pimsner-Popa constant.
\subsection{Completely positive bimodule maps}
In this subsection, we briefly investigate the completely positive bimodule maps for subfactors.
Suppose $\mathcal{A}$, $\mathcal{B}$ are unital C$^*$-algebras.
A linear map $\Phi:\mathcal{A}\to\mathcal{B}$ is positive if $\Phi(x^*x)\geq0$ for all $x\in\mathcal{A}$.
Suppose $n\in \mathbb{N}$.
A linear map $\Phi:\mathcal{A}\to\mathcal{B}$ is $n$-positive if $\Phi\otimes id_n$ is positive on $\mathcal{A}\otimes M_n(\mathbb{C})$.
A linear map $\Phi:\mathcal{A} \to \mathcal{B}$ is completely positive if it is $n$-positive for all $n\in\mathbb{N}$.
When $\mathcal{A}, \mathcal{B}$ are von Neumann algebras, we assume that $\Phi$ is normal.
A normal completely positive map $\Phi$ is a quantum channel if $\Phi$ is unital, i.e., $\Phi(\mathbf{1})=\mathbf{1}$.

Suppose that $\mathcal{N}\subseteq\mathcal{M}$ is a subfactor.
A normal linear map $\Phi:\mathcal{M}\to \mathcal{M}$ is an $\mathcal{N}$-$\mathcal{N}$-bimodule map if 
\begin{align*}
    \Phi(y_1 x y_2)=y_1\Phi(x) y_2, \quad\text{for all } x\in\mathcal{M},\ y_i\in\mathcal{N},\ i=1,2.
\end{align*}
We denote the set of all normal completely positive bimodule (CPB) maps by $\mathbf{CP}_{\mathcal{N}}(\mathcal{M})$ and the set of all bounded bimodule maps by $\CB$.
We say $\Phi\in\mathbf{CP}_{\mathcal{N}}(\mathcal{M})$ a bimodule quantum channel if $\Phi(\mathbf{1})=\mathbf{1}$.

We next show that $\mathbf{CP}_{\mathcal{N}}(\mathcal{M})$ and $\mathcal{N}'\cap \mathcal{M}_1$ coincide.
For any $\Phi\in \CB$, $\Phi$ commutes with the left and right actions of $\mathcal{N}$ on $L^2(\mathcal{M})$.
Hence $\Phi$ is corresponding to an element $y_{\Phi}\in\mathcal{N}'\cap \mathcal{M}_1$ satisfying 
\begin{align}\label{eq:bimodule map correponding to operator}
y_{\Phi} x\Omega=\Phi(x)\Omega\quad \text{ for all } x\in \mathcal{M}.
\end{align}
On the other hand, for any $y\in\mathcal{N}'\cap \mathcal{M}_1$, we define
\begin{align}
      \Phi(x)=\mu E_{\mathcal{M}}(yxe_1),\quad \text{ for all } x\in\mathcal{M}.
\end{align}
One can easily check that $\Phi\in\CB$.
Therefore, we shall identify $\CB$ with $\mathcal{N}'\cap \mathcal{M}_1$ and the bimodule $_{\mathcal{N}}\mathcal{M}_{\mathcal{N}}$ (see also \cite{Bis97}).
\subsection{Fourier multipliers and  pictorial interpretations}
To introduce the Fourier  multiplier  of $\Phi$, we first recall the following zig and zag operators defined in \cite[Definition 4.1.9]{Jon21}:
\begin{align*}
    \kappa_1(x_1\otimes x_2)= &\mu^{-1/2} \sum_{j=1}^m x_1\eta_j^* \otimes \eta_j \otimes x_2 \\
    \mu_3(x_1\otimes x_2\otimes x_3)=& x_1\otimes x_2x_3,
\end{align*}
where $x_1\otimes x_2\in \bigotimes_{\mathcal{N}}^2\mathcal{M}$ and $x_1\otimes x_2\otimes x_3\in \bigotimes_{\mathcal{N}}^3\mathcal{M}$.
Hence the inverse Fourier transform of $y_{\Phi}$ in $\bigotimes_{\mathcal{N}}^2\mathcal{M}$ becomes 
\begin{align*}
    \mu_3(\id \otimes \Phi \otimes \id )\kappa_1(x\otimes y)
    =&\sum_{j=1}^m\mu_3 (\id \otimes \Phi \otimes \id ) x\eta_j^* \otimes \eta_j \otimes y\\
    =& \sum_{j=1}^m\mu_3 x\eta_j^* \otimes \Phi(\eta_j) \otimes y\\
    =& \sum_{j=1}^m x\eta_j^* \otimes \Phi(\eta_j)y,
\end{align*}
where $x, y\in \mathcal{M}$.
\begin{definition}[Fourier multipliers]
   For any $\Phi\in\CB$, we call the inverse Fourier transform of $y_{\Phi}$ the {\sl Fourier multiplier} of $\Phi$ and denote it by $\widehat{\Phi}$.
\end{definition}
We note that $\widehat{\Phi}\in \mathcal{M}'\cap \mathcal{M}_2$.
By the Hausdorff-Young inequality \cite{JLW16}:
\begin{align*}
    \|\widehat{\Phi}\|_{\infty}\leq\mu^{-1/2} \|y_{\Phi}\|_1,
\end{align*}
we see that $\widehat{\Phi}$ is bounded.
Let $\Omega_1$ be the cyclic and separating vector in $L^2(\mathcal{M}_1)$ for $\mathcal{M}_1$.
Identifying $\widehat{\Phi}$ as an element in $\mathcal{M}'\cap \mathcal{M}_2$, we obtain that
    \begin{align}\label{eq:fouriermultiple}
        \widehat{\Phi}  x e_1 y \Omega_1 = \mu^{1/2}\sum_{j=1}^m x \eta_j^* e_1\Phi(\eta_j)y\Omega_1, \quad \text{ for all } x, y\in\mathcal{M}.
    \end{align}

Next we show the pictorial interpretation of the Fourier multiplier in Jones planar algebras \cite{BisJon00,Jon21}. 
Suppose that $\mathscr{P}=\{\mathscr{P}_{n,\pm}\}_{n\geq0}$ is a subfactor planar algebra.
For the two box spaces, we have $\mathscr{P}_{2,+}\cong \mathcal{N}'\cap\mathcal{M}_1$ and $\mathscr{P}_{2,-}\cong \mathcal{M}'\cap\mathcal{M}_2$.
The Fourier transform $\mathfrak{F}$: $\mathscr{P}_{2,+}\to\mathscr{P}_{2,-}$ is a $90^\circ$ rotation:
\begin{align}\label{eq:Fourier transform}
\mathfrak{F}(x):=\raisebox{-0.9cm}{
\begin{tikzpicture}[scale=1.5]
\path [fill=gray!40] (-0.3, -0.4) rectangle (0.8, 0.9);
\path [fill=white] (0.35, -0.4)--(0.35, 0.5) .. controls +(0, 0.3) and +(0, 0.3) .. (0.65, 0.5)--(0.65, -0.4);
\path[fill=white] (0.15, 0.9) -- (0.15, 0) .. controls +(0, -0.3) and +(0, -0.3) .. (-0.15, 0)--(-0.15, 0.9);
\draw [blue, fill=white] (0,0) rectangle (0.5, 0.5);
\node at (0.25, 0.25) {$x$};
\draw (0.35, 0)--(0.35, -0.4) (0.15, 0.5)--(0.15, 0.9);
\draw (0.35, 0.5) .. controls +(0, 0.3) and +(0, 0.3) .. (0.65, 0.5)--(0.65, -0.4);
\draw (0.15, 0) .. controls +(0, -0.3) and +(0, -0.3) .. (-0.15, 0)--(-0.15, 0.9);
\end{tikzpicture}}
\;.
\end{align}
The Fourier transform  induces the convolution on $\mathscr{P}_{2,+}$:
\begin{align}\label{eq:convolution1}
   x\ast y:=\mathfrak{F}^{-1}(\mathfrak{F}(y)\mathfrak{F}(x))=\raisebox{-0.9cm}{
\begin{tikzpicture}[scale=1.5]
\path [fill=gray!40] (0.15, -0.4) rectangle (1.05, 0.9);
\path [fill=white] (0.35, 0.5) .. controls +(0, 0.3) and +(0, 0.3) .. (0.85, 0.5)-- (0.85, 0) .. controls +(0, -0.3) and +(0, -0.3) .. (0.35, 0)--(0.35, 0.5);
\draw [blue, fill=white] (0,0) rectangle (0.5, 0.5);
\node at (0.25, 0.25) {$x$};
\draw (0.15, 0.5)--(0.15, 0.9) (0.15, 0)--(0.15, -0.4);
\begin{scope}[shift={(0.7, 0)}]
\draw [blue, fill=white] (0,0) rectangle (0.5, 0.5);
\node at (0.25, 0.25) {$y$};
\draw (0.35, 0.5)--(0.35, 0.9)  (0.35, 0)--(0.35, -0.4);
\end{scope}
\draw (0.35, 0.5) .. controls +(0, 0.3) and +(0, 0.3) .. (0.85, 0.5);
\draw (0.35, 0) .. controls +(0, -0.3) and +(0, -0.3) .. (0.85, 0);
\end{tikzpicture}}
\;.
\end{align}
Suppose $\Phi\in\CB$, we have the pictorial interpretation of the Fourier multiplier $\widehat{\Phi}\in\mathscr{P}_{2,-}$ as the Fourier inverse of $y_{\Phi}\in \mathscr{P}_{2,+}$:
\begin{align*}
    \widehat{\Phi}=\raisebox{-0.6cm}{
\begin{tikzpicture}[xscale=1.2]
\path [fill=gray!40] (-0.4, -0.4) rectangle (0.9, 0.9);
\path [fill=white] (0.35, 0.9)--(0.35, 0) .. controls +(0, -0.3) and +(0, -0.3) .. (0.65, 0)--(0.65, 0.9);
\path[fill=white] (0.15, -0.4) -- (0.15, 0.5) .. controls +(0, 0.3) and +(0, 0.3) .. (-0.15, 0.5)--(-0.15, -0.4);
\draw (0.35, 0.5)--(0.35, 0.9) (0.15, 0)--(0.15, -0.4);
\draw  (0.35, 0.9)--(0.35, 0) .. controls +(0, -0.3) and +(0, -0.3) .. (0.65, 0)--(0.65, 0.9);
\draw (0.15, -0.4) -- (0.15, 0.5) .. controls +(0, 0.3) and +(0, 0.3) .. (-0.15, 0.5)--(-0.15, -0.4);
\draw [blue, fill=white] (0,0) rectangle (0.5, 0.5);
\node at (0.25, 0.25) {\tiny $y_{\Phi}$};
\end{tikzpicture}}\;.
\end{align*}
The information of $\Phi$ is completely determined by $\widehat{\Phi}$.

\begin{remark}
The pictorial interpretation of $\mu_3(\id \otimes \Phi \otimes \id )\kappa_1$ is
\begin{align*}
    \mu_3(\id \otimes \Phi \otimes \id )\kappa_1=\raisebox{-0.6cm}{
\begin{tikzpicture}[xscale=1.2]
\path [fill=gray!40] (-0.4, -0.4) rectangle (0.9, 0.9);
\path [fill=white] (0.35, 0.9)--(0.35, 0) .. controls +(0, -0.3) and +(0, -0.3) .. (0.65, 0)--(0.65, 0.9);
\path[fill=white] (0.15, -0.4) -- (0.15, 0.5) .. controls +(0, 0.3) and +(0, 0.3) .. (-0.15, 0.5)--(-0.15, -0.4);
\draw (-0.4, -0.4)--(-0.4, 0.9) (0.9, -0.4)--(0.9, 0.9);
\draw (0.35, 0.5)--(0.35, 0.9) (0.15, 0)--(0.15, -0.4);
\draw  (0.35, 0.9)--(0.35, 0) .. controls +(0, -0.3) and +(0, -0.3) .. (0.65, 0)--(0.65, 0.9);
\draw (0.15, -0.4) -- (0.15, 0.5) .. controls +(0, 0.3) and +(0, 0.3) .. (-0.15, 0.5)--(-0.15, -0.4);
\draw [blue, fill=white] (0,0) rectangle (0.5, 0.5);
\node at (0.25, 0.25) {\tiny $y_{\Phi}$};
\end{tikzpicture}}\;.
\end{align*}
\end{remark}
\subsection{Explicit formulas for CPB maps}
In this subsection, we see two algebraic expressions for bounded bimodule maps on the subfactors.
\begin{proposition}\label{prop:phi formula}
  Suppose $\mathcal{N}\subseteq\mathcal{M}$ is a subfactor of finite index $\mu$, $\mu>0$, and $\Phi$ is a bounded bimodule map on $\mathcal{M}$. 
  Then we have 
  \begin{align}\label{eq:expression bimodule map}
     \Phi(x)=\mu E_{\mathcal{M}}(y_{\Phi}xe_1),\quad x\in\mathcal{M}.
  \end{align}
\end{proposition}
\begin{proof}
Note that for any $x\in \mathcal{M}$, we have
    \begin{align*}
        \Phi(x)\Omega=y_{\Phi} x\Omega
        =\mu E_{\mathcal{M}}(y_{\Phi} xe_1)\Omega.
    \end{align*}
Then the proposition is true.
\end{proof}
\begin{remark}
We also study the case when the subfactor $\mathcal{N}\subseteq\mathcal{M}$ is non-extremal.
    Note that $\tau_{\mathcal{N}'}$ and $\tau_2$ are two faithful traces on $\mathcal{N}'\cap\mathcal{M}_2$.
Then there exists a strictly positive operator $\Delta$ in the center $Z(\mathcal{N}'\cap\mathcal{M}_2)$ such that
 $\tau_2(x)=\tau_{\mathcal{N}'}(\Delta^2 x)$ and $\tau_{\mathcal{N}'}(x)=\tau_2(\Delta^{-2} x)$ for all $x\in\mathcal{N}'\cap\mathcal{M}_2$.
 The Fourier transform $\mathfrak{F}$: $\mathcal{N}'\cap \mathcal{M}_1\to\mathcal{M}'\cap\mathcal{M}_2$ and its inverse $\mathfrak{F}^{-1}$: $\mathcal{M}'\cap\mathcal{M}_2\to\mathcal{N}'\cap \mathcal{M}_1$ are defined as follows:
    \begin{align*}
    \mathfrak{F}(x)=&\mu^{3/2} E_{\mathcal{M}'\cap\mathcal{M}_2}(\Delta xe_2e_1), \quad x\in \mathcal{N}'\cap \mathcal{M}\\
     \mathfrak{F}^{-1}(x)=& \mu^{3/2} E_{\mathcal{M}_1}(\Delta^{-1} x e_1e_2) \quad x\in \mathcal{M}'\cap \mathcal{M}_2.
     \end{align*}
     For any bimodule map $\Phi$ on $\mathcal{M}$, we then have the following expression:
  \begin{align*}
     \Phi(x)=\mu^2 E_{\mathcal{M}}(\Delta e_2) E_{\mathcal{M}}(\mathfrak{F}^{-1}(\widehat{\Phi})x e_1),\quad x\in\mathcal{M}.
  \end{align*}
  When $\Delta=\mathbf{1}$, by $ E_{\mathcal{M}}(e_2)=\mu^{-1}\mathbf{1}$, we recover Equation \eqref{eq:expression bimodule map}.
\end{remark}

\begin{remark}\label{rem:bimodule map}
Suppose $\mathcal{N}\subseteq\mathcal{M}$ is a subfactor.
For any $y\in\mathcal{N}'\cap\mathcal{M}_1$, we define a linear map $\Theta_y$ on $\mathcal{M}$ as
\begin{align}
 \Theta_y(x)=\mu E_{\mathcal{M}}(yxe_1),\quad x\in\mathcal{M}.
\end{align}
In particular, $\Theta_{\mathbf{1}}=id$ and $\Theta_{e_1}=E_{\mathcal{N}}$.
It is clear that $\Theta_y$ is a completely bounded bimodule map.
We say $\Theta_y$ is the bimodule map on $\mathcal{M}$ induced by $y$.
By the uniqueness of the Fourier multiplier of $\Theta_y$, we have that
\begin{align}\label{eq:for1}
    \widehat{\Theta_y}=\mathfrak{F}^{-1}(y).
\end{align}
Hence for any $\Phi\in \CB$, we have that
\begin{align}\label{eq:for2}
    \Phi=\Theta_{\mathfrak{F}(\widehat{\Phi})}=\Theta_{y_{\Phi}}.
\end{align}
\end{remark}
\begin{remark}\label{rem:operation bimodule map}
Suppose $\Phi_1, \Phi_2\in \CB$ and $y_j=\mathfrak{F}(\widehat{\Phi_j})=y_{\Phi_j}$ for $j=1,2$.
It is known that the addition, the composition and the involution of bimodule maps are described as follows:
\begin{align*}
    \Phi_1+\Phi_2=& \Theta_{y_1+y_2}, \\
    \Phi_1\Phi_2=& \Theta_{y_1y_2},\\
    \Phi_1^*=&\Theta_{y_1^*}.
\end{align*}
Due to the convolution on $\mathcal{N}'\cap \mathcal{M}_1$, we can define an operation $\star$ on bimodule maps as follows:
\begin{align*}
    \Phi_1\star\Phi_2=\Theta_{y_1* y_2},
\end{align*}
where
\begin{align*}
    y_1*y_2=\mathfrak{F}^{-1}(\mathfrak{F}(y_1)\mathfrak{F}(y_2)).
\end{align*}
\end{remark}
To see the equivalence between the complete positivity of $\Phi$ and the positivity of $\widehat{\Phi}$, we need the following characterization of $\Phi$.
\begin{proposition}\label{prop:formula}
 Suppose $\mathcal{N}\subseteq\mathcal{M}$ is a subfactor of finite index $\mu$, $\mu>0$, and $\Phi$ is a bounded bimodule map. 
Then 
\begin{align}\label{eq:phi}
    \Phi(x)=\mu^{3/2} E_{\mathcal{M}}(e_2e_1\widehat{\Phi} x e_1e_2).
\end{align}
\end{proposition}
\begin{proof}
By Equation \eqref{eq:fouriermultiple}, we have
\begin{align*}
    e_1\Phi(x)\Omega_1=& \sum_{j=1}^m e_1 \Phi(E_{\mathcal{N}} (x\eta_j^*)\eta_j)\Omega_1 \\
    =& \sum_{j=1}^m  e_1 E_{\mathcal{N}} (x\eta_j^*)\Phi(\eta_j)\Omega_1 \\
    =& \sum_{j=1}^m  e_1 x\eta_j^* e_1\Phi(\eta_j)\Omega_1 \\
    =&\mu^{-1/2} e_1\widehat{\Phi} x e_1e_2\Omega_1.
\end{align*}
Note that there exists $y\in \mathcal{M}_1$ such that $y e_2=\mu^{-1/2} e_1\widehat{\Phi} x e_1e_2$.
We then we have $e_1\Phi(x)\Omega_1=y\Omega_1$
Note that $\Omega_1$ is separating for $\mathcal{M}_1$.
We obtain that $e_1\Phi(x)=y$.
Then
\begin{align*}
e_1\Phi(x)e_2=ye_2=\mu^{-1/2} e_1\widehat{\Phi} x e_1e_2.
\end{align*}
Multiplying $e_2$ from the left-hand side, we have
\begin{align*}
   \mu^{-1} e_2\Phi(x)=\mu^{-1/2} e_2e_1\widehat{\Phi} x e_1e_2. 
\end{align*}
By taking the conditional expectation $E_{\mathcal{M}}$, we see Equation \eqref{eq:phi} is true.
\end{proof}

\begin{remark}
    By Proposition \ref{prop:formula}, we see that every bounded bimodule map on $\mathcal{N}\subseteq\mathcal{M}$ is completely bounded.
\end{remark}

\begin{remark}
In terms of the bimodule language, we have the following informally pictorial interpretation:
\begin{align}\label{eq:convolution picture}
    \Phi(y)=\raisebox{-0.9cm}{
\begin{tikzpicture}[scale=1.5]
\path [fill=gray!40] (0.5, -0.4) rectangle (0.8, 0.9);
\path [fill=gray!40] (-0.3, 0) rectangle (0, 0.6);
\draw [blue,fill=white] (0,0) rectangle (0.7, 0.5);
\node at (0.35, 0.25) {$\widehat{\Phi}$};
\begin{scope}[shift={(-0.6, 0)}]
\draw [blue,fill=white, rounded corners] (0,0) rectangle (0.35, 0.5);
\node at (0.175, 0.25) {$y$};
\end{scope}
\draw [fill=gray!40] (0.5, 0)--(0.5, -0.4) (0.5, 0.5)--(0.5, 0.9);
\draw[fill=gray!40] (0.2, 0.5) .. controls +(0, 0.4) and +(0, 0.4) .. (-0.425, 0.5);
\draw[fill=gray!40] (0.2, 0) .. controls +(0, -0.4) and +(0, -0.4) .. (-0.425, 0);
\end{tikzpicture}},
\end{align}
where $y\in \mathcal{M}$ and the box of $y$ is depicted as  a rounded box because it is not in the planar algebra $\mathscr{P}^{\mathcal{N}\subseteq \mathcal{M}}$, but it acts like a 1-box element.
This means $\Phi(y)=y*\widehat{\Phi}$ informally.
When the inclusion $\mathcal{N}\subseteq \mathcal{M}$ is $\mathbb{C}\subseteq M_n(\mathbb{C})$, by using the spin model, we have the following spectral decomposition of $\widehat{\Phi}$:
\begin{align*}
\vcenter{\hbox{
\begin{tikzpicture}[scale=1.5]
\path [fill=gray!40] (0.5, -0.4) rectangle (0.8, 0.9);
\path [fill=gray!40] (0.2, -0.4) rectangle (-0.1, 0.9);
\draw (0.5, -0.4)--(0.5, 0.9) (0.2, -0.4)--(0.2, 0.9);
\draw [blue,fill=white] (0,0) rectangle (0.7, 0.5);
\node at (0.35, 0.25) {$\widehat{\Phi}$};
\end{tikzpicture}}}
=\sum_j \vcenter{\hbox{
\begin{tikzpicture}[scale=1.5]
\path [fill=gray!40] (0.3,0.7) .. controls +(0, -0.4) and +(0, -0.4) .. (1, 0.7)--(1, -0.2)..  controls +(0, 0.4) and +(0, 0.4).. (0.3,-0.2) --(0.3, 0.7);
\path [fill=gray!40] (0, 1.2) rectangle (0.3, -0.8);
\path [fill=gray!40] (1, 1.2) rectangle (1.3, -0.8);
\draw (0.3,1.2)--(0.3,0.7) .. controls +(0, -0.4) and +(0, -0.4) .. (1, 0.7)--(1, 1.2);
\draw (0.3, -0.8)--(0.3,-0.2) .. controls +(0, 0.4) and +(0, 0.4) .. (1, -0.2)--(1, -0.8);
\draw [fill=white] (0.1,0.7) rectangle (0.5, 1.1);
\node at (0.3, 0.9) {\tiny $K_j$};
\draw [fill=white] (0.1,-0.2) rectangle (0.5, -0.6);
\node at (0.3, -0.4) {\tiny $K_j^*$};
\end{tikzpicture}}}.
\end{align*}
This gives the Kraus decomposition of $\Phi$.
In general, Equation \eqref{eq:phi} together with the spectral decomposition of $\widehat{\Phi}$ gives a Kraus decomposition of $\Phi$. 
Informally, we have
\begin{align*}
        \Phi(y)=\sum_i
        \vcenter{\hbox{
\begin{tikzpicture}[scale=1.5]
\path [fill=gray!40] (0.5, -0.4) rectangle (0.8, 0.9);
\path [fill=gray!40] (-0.3, 0) rectangle (0, 0.6);
\begin{scope}[shift={(-0.6, 0)}]
\draw [blue,fill=white, rounded corners] (0,0) rectangle (0.35, 0.5);
\node at (0.175, 0.25) {$y$};
\end{scope}
\draw [fill=gray!40] (0.5, 0)--(0.5, -0.4) (0.5, 0.5)--(0.5, 0.9);
\draw[fill=gray!40] (0.2, 0.5) .. controls +(0, 0.4) and +(0, 0.4) .. (-0.425, 0.5);
\draw[fill=gray!40] (0.2, 0) .. controls +(0, -0.4) and +(0, -0.4) .. (-0.425, 0);
\path [fill=gray!40] (0, 0.5) rectangle (0.5, 0);
\draw [fill=white] (0.2, 0.5) .. controls +(0, -0.25) and +(0, -0.25) .. (0.5, 0.5);
\draw [fill=white] (0.2, 0) .. controls +(0, 0.25) and +(0, 0.25) .. (0.5, 0);
\draw [fill=white] (0, 0.4) rectangle (0.4, 0.6);
\node at (0.2, 0.5) {\tiny $K_i$};
\draw [fill=white] (0, -0.1) rectangle (0.4, 0.1);
\node at (0.2, 0) {\tiny $K_i^*$};
\end{tikzpicture}}}
        +\sum_j\vcenter{\hbox{
\begin{tikzpicture}[scale=1.5]
\path [fill=gray!40] (0.5, -0.4) rectangle (0.8, 0.9);
\path [fill=gray!40] (-0.3, 0) rectangle (0, 0.6);
\draw [blue,fill=white] (0,0) rectangle (0.7, 0.5);
\node at (0.35, 0.25) {$S_j$};
\begin{scope}[shift={(-0.6, 0)}]
\draw [blue,fill=white, rounded corners] (0,0) rectangle (0.35, 0.5);
\node at (0.175, 0.25) {$y$};
\end{scope}
\draw [fill=gray!40] (0.5, 0)--(0.5, -0.4) (0.5, 0.5)--(0.5, 0.9);
\draw[fill=gray!40] (0.2, 0.5) .. controls +(0, 0.4) and +(0, 0.4) .. (-0.425, 0.5);
\draw[fill=gray!40] (0.2, 0) .. controls +(0, -0.4) and +(0, -0.4) .. (-0.425, 0);
\end{tikzpicture}}},
\end{align*}
where $\{S_j\}_j$ is an orthogonal family of multiples of projections.
\end{remark}

\begin{proposition}\label{prop:cpmapeq}
 Suppose $\mathcal{N}\subseteq\mathcal{M}$ is a subfactor of finite index and $\Phi\in \CB$.
Then $\Phi$ is completely positive if and only if $\widehat{\Phi}$ is positive.
\end{proposition}
\begin{proof}
Suppose that $\widehat{\Phi}$ is positive. 
By Proposition \ref{prop:formula}, we see that 
\begin{align*}
    \Phi(x)=\mu^{3/2} E_{\mathcal{M}}(e_2e_1 \widehat{\Phi}^{1/2} x \widehat{\Phi}^{1/2} e_1e_2), \quad \text{ for all } x\in \mathcal{M}.
\end{align*}
This implies that $\Phi$ is completely positive.

Suppose $\Phi$ is completely positive. 
For any $\displaystyle \xi=\sum_{i=1}^m x_ie_1 y_i$, we have
\begin{align*}
\langle \widehat{\Phi}\xi,\xi\rangle
=&\sum_{i,j=1}^m \sum_{k=1}^n \tau (E_{\mathcal{N}}(x_j^*x_i\eta_k^*)\Phi(\eta_k)y_iy_j^* )\\
=&\sum_{i,j=1}^m \tau \left(\Phi\left(\sum_{k=1}^nE_{\mathcal{N}}(x_j^*x_i\eta_k^*)\eta_k\right)y_iy_j^*\right)\\
=&\sum_{i,j=1}^m \tau (\Phi(x_j^*x_i)y_iy_j^*)
\geq 0.
\end{align*}
The last inequality is due to the complete positivity of $\Phi$.
\end{proof}
\begin{remark}
    Suppose $\mathcal{N}\subseteq\mathcal{M}$ is a subfactor of finite index and $\mathscr{P}$ is the subfactor planar algebra.
A 2-box $y\in\mathscr{P}_{2,+}\cong \mathcal{N}'\cap \mathcal{M}_1$ is called $\mathfrak{F}$-positive if its Fourier transform $\mathfrak{F}(y)$ is positive (see \cite[Definition 2.6]{HJLW23}).
Since $\widehat{\Phi}=\mathfrak{F}^{-1}(y_{\Phi})$ and $\mathfrak{F}^2$ is an anti $*$-isomorphism, we see that $\Phi$ is completely positive  if and only if $y_{\Phi}$ is $\mathfrak{F}$-positive  by Proposition \ref{prop:cpmapeq}.
This $\mathfrak{F}$-positivity is a generalization of the complete positivity of a  quantum channel in quantum information. 
Proposition \ref{prop:cpmapeq} generalizes Choi's Theorem 2 of  \cite{Cho75}: a linear map $\Phi$ on  $M_n(\mathbb{C})$ is completely positive if and only if its Choi's matrix is positive-semidefinite.
\end{remark}

\begin{remark}\label{rem:examples}
By a direct computation, we have that $\widehat{id}=\mu^{3/2}e_2$ and $\widehat{E_{\mathcal{N}}}=\mu^{1/2}\mathbf{1}$.
\end{remark}

\subsection{Pimsner-Popa inequality}
In this subsection, we see the Pimsner-Popa inequality for completely positive bimodule maps.
\begin{lemma} \label{lem:exists}
Suppose $\mathcal{N}\subseteq\mathcal{M}$ is a subfactor of finite index $\mu$, $\mu>0$, and $\Phi\in \CB$. 
Then we have that 
\begin{align}\label{eq:four3}
   \left|(\Phi(\eta_k\eta_j^*))_{j,k=1}^m\right| \leq \left\| (\Phi(\eta_k\eta_j^*))_{j,k=1}^m \right\| (E_{\mathcal{N}}(\eta_k\eta_j^*))_{j,k=1}^m.
\end{align}
Consequently, 
\begin{align*}
    \|\widehat{\Phi}\|\leq \mu^{1/2}\left\| (\Phi(\eta_k\eta_j^*))_{j,k=1}^m \right\|.
\end{align*}
\end{lemma}
\begin{proof}
Note that $E_{\mathcal{N}}(\eta_k \eta_j^*)=\delta_{j,k}$ for $j\neq k$ and $E_{\mathcal{N}}(\eta_j \eta_j^*)$ is a projection for $j=1,\ldots, m$.
Recalling that $\eta_j^*e_1$ is a partial isometry in $\mathcal{M}_1$, we have
\begin{align*}
    \eta_j^* e_1 \eta_j\eta_j^* e_1=\eta_j^* e_1.
\end{align*}
This implies that $\eta_j^*E_{\mathcal{N}}(\eta_j\eta_j^*)=\eta_j^*$.
Hence 
\begin{align*}
\Phi(\eta_j\eta_j^*)= \Phi(\eta_j\eta_j^*E_{\mathcal{N}}(\eta_j\eta_j^*)) = \Phi(\eta_j\eta_j^*)E_{\mathcal{N}}(\eta_j\eta_j^*).
\end{align*}
We then see that the support of $ (\Phi(\eta_k\eta_j^*))_{j,k=1}^m$ is contained in the support of $(E_{\mathcal{N}}(\eta_k\eta_j^*))_{j,k=1}^m$ and Equation \eqref{eq:four3} is true.
Note that
\begin{align*}
    \left\| \sum_i\sum_{j=1}^m x_i \eta_j^* e_1\Phi(\eta_j)y_i\Omega_1\right\|
    \leq \left\| (\Phi(\eta_k\eta_j^*))_{j,k=1}^m \right\| \left\| \sum_{i}x_ie_1y_i\Omega_1\right\|.
\end{align*}
We see that
\begin{align*}
\|\widehat{\Phi}\|\leq \mu^{1/2}\left\| (\Phi(\eta_k\eta_j^*))_{j,k=1}^m \right\|.
\end{align*}
This completes the proof.
\end{proof}

\begin{corollary}[Pimsner-Popa Inequality for CPB maps]
Suppose $\mathcal{N}\subseteq \mathcal{M}$ is a subfactor and $\Phi\in \CP$. Then
    \begin{align}\label{eq:pimpopineq}
    \Phi\leq \left\| (\Phi(\eta_k\eta_j^*))_{j,k=1}^m \right\| E_{\mathcal{N}},
    \end{align}
    i.e., $\displaystyle \left\|(\Phi(\eta_k\eta_j^*))_{j,k=1}^m \right\| E_{\mathcal{N}}-\Phi$ is completely positive.
Furthermore,
    \begin{align}\label{eq:pimpopineq2}
        \Phi\leq \mu\left\| \Phi \right\|_{cb} E_{\mathcal{N}},
    \end{align}
    where $\|\Phi\|_{cb}$ is the completely bounded norm.
\end{corollary}
\begin{proof}
    By Equality \eqref{eq:four3}, we see that Equation \eqref{eq:pimpopineq} is true and the constant is optimal.
    Note that
        \begin{align*}
        \|(\eta_k\eta_j^*)_{j,k=1}^m\|=\left\|\sum_{j=1}^m \eta_j^*\eta_j\right\|=\mu.
    \end{align*}
    We see Equation \eqref{eq:pimpopineq2} is true.
\end{proof}

\begin{remark}
    We take $\Phi=id$ in Equation \eqref{eq:pimpopineq2}, then we obtain the Pimsner-Popa inequality \eqref{eq:PP inequality}.
\end{remark}

\begin{remark}\label{rem:order relation}
Suppose $\Phi_1, \Phi_2\in \CP$.
Suppose $c>0$.
By the proof of Lemma \ref{lem:exists},
we see that $\Phi_1\leq c\Phi_2$ if and only if $\widehat{\Phi}_1 \leq c \widehat{\Phi}_2$.
\end{remark}

\section{Eigenspaces for Phase Groups}\label{sec:PF theory for CPB}
In this section, we reformulate the Perron-Frobenius theorem for completely positive bimodule maps.
We describe the structures of the Perron-Frobenius spaces for CPB maps using a maximal-support eigenvector.
Furthermore, we characterize the eigenspaces for phase groups of bimodule quantum channels.
\subsection{Perron-Frobenius theorem for CPB maps}
We first reformulate the Perron-Frobenius theorem for completely positive bimodule maps.
Recall that the spectrum $\sigma(\Phi)$ of a bounded linear map $\Phi:\mathcal{M}\to \mathcal{M}$ is given by 
\begin{align*}
    \sigma(\Phi)=\{c\in\mathbb{C}: c\id-\Phi \text{ is not invertible.}\}
\end{align*}
We denote by $r(\Phi)$ the spectral radius of $\Phi$.
Note that when $\Phi$ is a bimodule map, the inverse of $c\id -\Phi$ is also a bimodule map for $c\notin \sigma(\Phi)$.

Recall that for any $\Phi\in\CB$, there exists an element $y_{\Phi}\in\mathcal{N}'\cap\mathcal{M}_1$ such that $\Phi(x)=\Theta_{y_{\Phi}}(x)$ for $x\in\mathcal{M}$.
We denote by $\sigma(y_{\Phi})$ the spectrum of $y_{\Phi}$.
The following lemma shows that the two spectrums $\sigma(\Phi)$ and $\sigma(y_{\Phi})$ coincide.
\begin{lemma}\label{lem:spectrum equal}
    Suppose  $\mathcal{N}\subseteq\mathcal{M}$ is a subfactor of finite index and $\Phi\in\CB$.
    Then 
    \begin{align*}
        \sigma(\Phi)=\sigma(y_\Phi).
    \end{align*}
\end{lemma}
\begin{proof}
    It follows from Remarks \ref{rem:bimodule map} and \ref{rem:operation bimodule map} that for any $c\in\mathbb{C}$,
    \begin{align*}
        c\id-\Phi=\Theta_{c\mathbf{1}-y_{\Phi}}.
    \end{align*}
    Hence $c\id-\Phi$ is invertible if and only if $c\mathbf{1}-y_{\Phi}$ invertible.
    We obtain the conclusion.
\end{proof}
\begin{remark}
    Since the subfactor $\mathcal{N}\subseteq\mathcal{M}$ is finite-index, $\mathcal{N}'\cap\mathcal{M}_1$ is a finite-dimensional $C^*$-algebra.
    Hence $\sigma(y_\Phi)$ is a finite set.
    By Lemma \ref{lem:spectrum equal}, $\sigma(\Phi)$ is also a finite set.
\end{remark}
 
\begin{proposition}\label{prop:PF theorem for bimodule map}
   Suppose  $\mathcal{N}\subseteq\mathcal{M}$ is a subfactor of finite index and $\Phi\in\mathbf{CP}_{\mathcal{N}}(\mathcal{M})$. 
Then there exists a nonzero $\Psi\in\mathbf{CP}_{\mathcal{N}}(\mathcal{M})$ such that
\begin{align}
        \Psi\Phi=\Phi\Psi=r\Psi,
    \end{align}
where $r$ is the spectral radius of $\Phi$.
\end{proposition}
\begin{proof}
By  \cite[Proposition 3.1]{HJLW23}, there exists a nonzero $\mathfrak{F}$-positive element $x\in\mathcal{N}'\cap\mathcal{M}_1$ such that
\begin{align*}
xy_\Phi=y_\Phi x=rx.
\end{align*}
Let $\Psi=\Theta_{x}$.
Then
\begin{align*}
    \Psi\Phi=\Theta_{x}\Theta_{y_\Phi}=\Theta_{x y_\Phi }=r\Theta_{x}=r\Psi.
\end{align*}
Similarly, $\Phi\Psi=r\Psi$.
We obtain the conclusion.
\end{proof}
\begin{corollary}\label{cor:PF theorem for bimodule map}
   Suppose  $\mathcal{N}\subseteq\mathcal{M}$ is a subfactor of finite index and $\Phi\in\mathbf{CP}_{\mathcal{N}}(\mathcal{M})$. 
Then there exists a nonzero positive element $x\in\mathcal{M}$ such that
\begin{align}
        \Phi(x)=rx,
    \end{align}
where $r$ is the spectral radius of $\Phi$.
\end{corollary}
\begin{proof}
    Taking $x=\Psi(\mathbf{1})$ in Proposition \ref{prop:PF theorem for bimodule map}, we can obtain the result.
\end{proof}

\subsection{Perron-Frobenius spaces for CPB maps}
Next we will show that the Perron-Frobenius space for a completely positive  bimodule map contains the multiplication of a maximal-support eigenvector and an intermediate von Neumann subalgebra of $\mathcal{N}\subseteq\mathcal{M}$.
We first see the notation for the Perron-Frobenius eigenspace.
\begin{notation}
Suppose $\mathcal{N}\subseteq\mathcal{M}$ is a subfactor and $\Phi$ is a bounded bimodule map.
Let $r$ be the spectral radius of $\Phi$.
We denote the Perron-Frobenius space of $\Phi$ by
\begin{align}
    \mathcal{E}(\Phi):=\{x\in\mathcal{M}:\ \Phi(x)=rx\}.
\end{align}
We further define
    \begin{align}
       p_{max}=\bigvee_{x\in\mathcal{E}(\Phi)^+}\mathcal{R}(x).
    \end{align}
\end{notation}
The following lemma gives the Perron-Frobenius eigenvector of a completely positive bimodule map, which has the maximal range projection. 
\begin{lemma}\label{lem:max support}
Suppose $\mathcal{N}\subseteq\mathcal{M}$ is a subfactor of finite index and $\Phi\in\CP$ with $r(\Phi)=1$.
Suppose $\Phi$ is contractive.
 Then the following limit exists in weak-operator topology:
    \begin{align}\label{eq:max PF vector}
\zeta:=\lim_{N\to\omega}\frac{1}{N}\sum_{i=0}^{N-1}\Phi^{i}(p_{max}),
\end{align}
where  $\omega\in\beta\mathbb{N}\setminus\mathbb{N}$ is a free ultrafilter. 
  Moreover, we have that $\zeta$ is a positive element in $\mathcal{N}'\cap \mathcal{M}$ satisfying that $\Phi(\zeta)=\zeta$ and $\mathcal{R}(\zeta)=p_{\max}$.

\end{lemma}
\begin{proof}
Since $\|\Phi\|_{\infty\to\infty}=\|\Phi(\mathbf{1})\|\leq 1$, we have $\|\Phi^{i}(p_{max})\|\leq1$.
Hence the sequence in equality \eqref{eq:max PF vector} is uniformly bounded and the limit exists in the weak-operator topology.
It is clear that $\zeta$ is positive.
Moreover,
\begin{align*}
    \Phi(\zeta)=\lim_{N\to\omega}\frac{1}{N}\sum_{i=0}^{N-1}\Phi^{i+1}(p_{max})=\zeta.
\end{align*}
Hence $\zeta\in\mathcal{E}(\Phi)^+$ and $\mathcal{R}(\zeta)\leq p_{\max}$.
For any $x\in\mathcal{E}(\Phi)^+$ with  $\mathcal{R}(x)=p$ and $\|x\|\leq1$, we have
\[\Phi^i(p_{max})\geq \Phi^i(p)\geq \Phi^i(x)=x. \]
Hence $\zeta\geq x$ and $\mathcal{R}(\zeta)\geq p$.
It follows that $\mathcal{R}(\zeta)\geq p_{max}$.
From the above, $\mathcal{R}(\zeta)= p_{max}$.

We next show that $\zeta$ commutes with all operators in $\mathcal{N}$.
For any positive operator $y\in\mathcal{N}$, 
\[\Phi(y\zeta y)=y\Phi(\zeta)y=y\zeta y.\]
Hence $\mathcal{R}(y\zeta y)\leq p_{max}$.
Thus $\mathcal{R}(y\zeta^{1/2})\leq p_{max}$.
We have $p_{max}y\zeta^{1/2}=y\zeta^{1/2}$.
It implies 
\[p_{max}yp_{max}=yp_{max}.\]
Therefore, 
\[p_{max}y=yp_{max}.\]
Then 
\[y\Phi^i(p_{max})=\Phi^i(yp_{max})=\Phi^i(p_{max}y)=\Phi^i(p_{max})y.\]
By the definition of $\zeta$, we obtain $y\zeta=\zeta y$.
Therefore, $\zeta\in\mathcal{N}'\cap\mathcal{M}$.

\end{proof}
Suppose $\mathcal{N}\subseteq\mathcal{M}$ is a subfactor and $\Phi\in \CP$.
To characterize the Perron-Frobenius space of $\Phi$, we introduce the following set:
    \begin{align}
        \mathcal{P}(\Phi):=\{x\in\mathcal{M}:\ xy_\Phi=y_\Phi x,\ x^*y_\Phi=y_\Phi x^*\}.
    \end{align}
   We see that $\mathcal{P}(\Phi)$ is an intermediate von Neumann subalgebra of the subfactor $\mathcal{N}\subseteq\mathcal{M}$.
In particular, when $\mathcal{N}\subseteq\mathcal{M}$ is irreducible, i.e., $\mathcal{N}'\cap\mathcal{M}=\mathbb{C}\mathbf{1}$, 
we have
$\mathcal{P}(\Phi)'\cap \mathcal{P}(\Phi)\subseteq\mathcal{N}'\cap\mathcal{M}=\mathbb{C}\mathbf{1}$.
Hence $\mathcal{P}(\Phi)$ is a factor.

The following proposition describes the connection between $\mathcal{E}(\Phi)$ and $\mathcal{P}(\Phi)$ subject to $\zeta$ described in Lemma \ref{lem:max support}.
\begin{proposition}\label{prop:PF space and max PF vector}
Suppose $\mathcal{N}\subseteq\mathcal{M}$ is a subfactor of finite index and $\Phi\in\mathbf{CP}_{\mathcal{N}}(\mathcal{M})$ with $r(\Phi)=1$.
Suppose $\Phi$ is contractive. 
Then for any $x\in \mathcal{P}(\Phi)$, we have $[\zeta,x]=0$ and
$$\Phi(\zeta x)=\zeta x.$$
Moreover,  for any factor $\mathcal{L}\subseteq \mathcal{P}(\Phi)$, if $\zeta \mathcal{L}\neq0$ then $\zeta x=0$ implies $x=0$ for any $x\in\mathcal{L}$.
    
\end{proposition}
\begin{proof}
 For any $x\in\mathcal{P}(\Phi)$, we have that
\begin{align*}
\Phi( x \zeta)
=&\mu E_{\mathcal{M}}(y_\Phi x \zeta  e_1 )\\
=&\mu x E_{\mathcal{M}}( y_\Phi \zeta e_1)\\
=&x\Phi(\zeta)= x \zeta,
\end{align*}
which implies $x \zeta \in\mathcal{E}(\Phi)$.

Let $x_p\in\mathcal{E}(\Phi)^+$ with the range projection $\mathcal{R}(x_p)=p$.
Then for any $x\in\mathcal{P}(\Phi)$,
\begin{align*}
\Phi(x^*x_px)=x^*x_px.
\end{align*}
Thus $\mathcal{R}(x^*x_p^{1/2})=\mathcal{R}(x^*x_px)\leq p_{max}$.
We obtain that $p_{max} x^* p=x^*p.$
By taking the union of $p$, we see that
$$p_{max} x^* p_{max}=x^*p_{max}.$$
Hence $p_{max}x=xp_{max}$ for any $x\in\mathcal{P}(\Phi)$.
Now for any $i\in\mathbb{N}$, we have 
\begin{align*}
\Phi^i(p_{max})x=\Phi^i(p_{max}x)=\Phi^i(x p_{max})=x\Phi^i(p_{max}).
\end{align*}
Thus $[\zeta,x]=0$ from Equation \eqref{eq:max PF vector}.

Suppose that $\zeta x=0$ for some $x\in\mathcal{L}$. 
Then
\begin{align}
  \mathcal{R}\left(  \zeta \bigvee_{u\in \mathcal{U}(\mathcal{L})}u^*\mathcal{R}(x)u\right)= \bigvee_{u\in \mathcal{U}(\mathcal{L})}u^*\mathcal{R}(\zeta x)u=0.
\end{align}
Note that $\displaystyle \bigvee_{u\in \mathcal{U}(\mathcal{L})}u^*\mathcal{R}(x)u=\mathbf{1}_{\mathcal{L}}$ if $\mathcal{R}(x)\neq 0$ as $\mathcal{N}$ is a factor. 
Since $\zeta \mathcal{L}\neq0$, we obtain that $\mathcal{R}(x)=0$.
\end{proof}

\subsection{Phase groups for bimodule quantum channels}
Next, we will characterize the eigenspaces for bimodule quantum channels on subfactors $\mathcal{N}\subseteq\mathcal{M}$.
Let $\Phi\in\mathbf{CP}_{\mathcal{N}}(\mathcal{M})$.
For any $\alpha\in\sigma(\Phi)$, we define
\begin{align*}
\mathcal{M}(\Phi,\alpha)=\{x\in\mathcal{M}:\ \Phi(x)=\alpha x\},
\end{align*}
which is the eigenspace of $\Phi$ corresponding to the eigenvalue $\alpha$.
By Lemma \ref{lem:spectrum equal}, $\alpha\in\sigma(y_{\Phi})$.
Then there exists a nonzero projection $p\in\mathcal{N}'\cap\mathcal{M}_1$ such that $y_{\Phi}p=\alpha p$.
By Remarks \ref{rem:bimodule map} and \ref{rem:operation bimodule map}, $\Phi \Theta_p=\Theta_{y_{\Phi}}\Theta_p=\Theta_{y_{\Phi}p}=\alpha\Theta_p$.
Since $\Theta_p$ is a nonzero bimodule map on $\mathcal{M}$, there exists $y\in\mathcal{M}$ such that $x:=\Theta_p(y)\neq0$.
Hence $\Phi(x)=\alpha x$ and $\mathcal{M}(\Phi,\alpha)$ is a non-empty set.
Let $r$  be the spectrum radius of $\Phi$.
We shall note that the two notations $\mathcal{M}(\Phi,r)$ and $\mathcal{E}(\Phi)$ are consistent.

On the other hand, we define
\begin{align*}
\mathcal{P}(\Phi,\alpha)=\{x\in\mathcal{M}:\ xy_\Phi^*=\alpha y_\Phi^* x, y_\Phi x=\alpha xy_\Phi\}.
\end{align*}
By Corollary \ref{cor:PF theorem for bimodule map}, we have $r\in \sigma(\Phi)$.
When  $r=1$, we have $\mathcal{P}(\Phi,1)=\mathcal{P}(\Phi).$

\begin{lemma}\label{lem:commute}
    Suppose $\mathcal{N}\subseteq\mathcal{M}$ is a subfactor of finite index and $\Phi\in\mathbf{CP}_{\mathcal{N}}(\mathcal{M})$ is contactive with $r(\Phi)=1$.
We assume there exists a faithful normal  state $\omega$ on $\mathcal{M}$ such that $\omega\Phi=\omega$.
Then for any $x\in\mathcal{M}(\Phi,\alpha)$, $|\alpha|=1$, we have
\begin{align*}
x\widehat{\Phi}^{1/2}e_1e_2=\alpha \widehat{\Phi}^{1/2}e_1e_2 x,\quad e_2e_1\widehat{\Phi}^{1/2}x=\alpha xe_2e_1\widehat{\Phi}^{1/2}. 
\end{align*}
\end{lemma}
\begin{proof}
By Kadison-Schwarz inequality, $\Phi(x^*x)\geq\Phi(x)^*\Phi(x)=x^*x$.
Since $\omega(\Phi(x^*x)-x^*x)=0$, we have $\Phi(x^*x)=x^*x$ as $\omega$ is faithful.
Let $T=x\widehat{\Phi}^{1/2}e_1e_2-\alpha \widehat{\Phi}^{1/2}e_1e_2 x$.
Then
\begin{align*}
\tau(E_{\mathcal{M}}(T^*T))=&\tau\left( E_{\mathcal{M}}[(x\widehat{\Phi}^{1/2}e_1e_2-\alpha \widehat{\Phi}^{1/2}e_1e_2 x)^*(x\widehat{\Phi}^{1/2}e_1e_2-\alpha \widehat{\Phi}^{1/2}e_1e_2 x)]\right)\\
=&\mu^{-3/2}\tau(\Phi(x^*x))+\mu^{-3/2}\tau(xx^*\Phi(\mathbf{1}))-\overline{\alpha}\mu^{-3/2}\tau(x^*\Phi(x))-\alpha\mu^{-3/2}\tau(x\Phi(x^*))\\
\leq&\mu^{-3/2}\tau(x^*x)+\mu^{-3/2}\tau(xx^*)-\mu^{-3/2}\tau(x^*x)-\mu^{-3/2}\tau(xx^*)\\
=&0.
\end{align*}
Hence $E_{\mathcal{M}}(T^*T)=0$.
By the Pimsner-Popa inequality, $T=0$.
Similarly, we can prove the second equality.
\end{proof}

Now we are able to describe the phase groups of bimodule maps and the associated eigenspaces.

\begin{theorem}\label{thm:entry}
Suppose $\mathcal{N}\subseteq\mathcal{M}$ is a subfactor of finite index and $\Phi$ is a bimodule quantum channel.
We assume that there exists a faithful normal  state $\omega$ on $\mathcal{M}$ such that $\omega\Phi=\omega$.
Then the following statements hold.
\begin{enumerate}[(i)]
\item  $\mathcal{M}(\Phi,\alpha)=\mathcal{P}(\Phi,\alpha)$ for any $\alpha\in\sigma(\Phi)\cap U(1)$, where $U(1)$ is the unit circle.
\item $\mathcal{M}(\Phi,\alpha)^*=\mathcal{M}(\Phi,\overline{\alpha})$ for any $\alpha\in\sigma(\Phi)\cap U(1)$.
\item $\mathcal{M}(\Phi,\alpha_1)\mathcal{M}(\Phi,\alpha_2)\subseteq \mathcal{M}(\Phi,\alpha_1\alpha_2)$ for any $\alpha_1,\alpha_2\in\sigma(\Phi)\cap U(1)$.
In particular, $\mathcal{M}(\Phi,\alpha)\mathcal{M}(\Phi,\overline{\alpha})$ is a $*$- subalgebra of $\mathcal{M}(\Phi,1)$.

 We further assume $\mathcal{M}(\Phi,1)$ is a factor.
\item The set $\Gamma:=\sigma(\Phi)\cap U(1)$ is a finite cyclic group.
\item
Then there exists a unitary $u_{\alpha}\in\mathcal{M}(\Phi,\alpha)$, $\alpha\in\Gamma$, such that 
\begin{align*}
\mathcal{M}(\Phi,\alpha)=u_{\alpha}\mathcal{M}(\Phi,1)=\mathcal{M}(\Phi,1)u_{\alpha}.
\end{align*}
Moreover, $\mathcal{M}(\Phi,\alpha)^{n}=\mathcal{M}(\Phi,\alpha^n)$, $n\in\mathbb{N}$.
In particular, $\mathcal{M}(\Phi,\alpha)^{|\Gamma|}=\mathcal{M}(\Phi,1)$, where $|\Gamma|$ is the order of the finite cyclic group $\Gamma$. 
This implies that $\mathcal{M}(\Phi, \alpha)$ is invertible $\mathcal{M}(\Phi,1)$-$\mathcal{M}(\Phi,1)$-bimodule, and the eigenspaces form a bimodule category which is a unitary fusion category.
\end{enumerate}
\end{theorem}
\begin{proof}
(i) By Lemma \ref{lem:commute}, we have $x\widehat{\Phi}e_1e_2=\alpha \widehat{\Phi}e_1e_2 x$ and $e_2e_1\widehat{\Phi}^{1/2}x=\alpha xe_2e_1\widehat{\Phi}^{1/2}$ for any $x\in\mathcal{M}(\Phi,\alpha)$. 
Taking the conditional expectation $E_{\mathcal{M}_1}$ on both sides of these two equalities, we obtain that $xy_\Phi^*=\alpha y_\Phi^* x$ and $y_\Phi x=\alpha xy_\Phi$.
Thus $\mathcal{M}(\Phi,\alpha)\subseteq\mathcal{P}(\Phi,\alpha)$.
On the other hand, for any $x\in\mathcal{P}(\Phi,\alpha)$,
\begin{align*}
\Phi(x)&=\mu E_{\mathcal{M}}(y_\Phi xe_1)=\alpha \mu E_{\mathcal{M}}(xy_\Phi e_1)\\
&=\alpha \mu x E_{\mathcal{M}}(y_\Phi e_1)=\alpha x\Phi(\mathbf{1})\\
&=\alpha x.
\end{align*}
Hence $\mathcal{P}(\Phi,\alpha)\subseteq \mathcal{M}(\Phi,\alpha)$.
From the above discussion, we obtain that $\mathcal{M}(\Phi,\alpha)=\mathcal{P}(\Phi,\alpha)$.

(ii) It follows from $\Phi(x)^*=\Phi(x^*)$ for any $x\in\mathcal{M}$.

(iii) For any $x_i\in\mathcal{M}(\Phi,\alpha_i)$, $i=1,2$, by (i) we have
\begin{align*}
x_1x_2y_\Phi^*&=\alpha_2 x_1y_\Phi^* x_2=\alpha_1\alpha_2y_\Phi^* x_1x_2,\\
y_\Phi x_1x_2&=\alpha_1 x_1y_\Phi x_2=\alpha_1\alpha_2 x_1x_2y_\Phi.
\end{align*}
It implies that $x_1x_2\in\mathcal{P}(\Phi,\alpha_1\alpha_2)=\mathcal{M}(\Phi,\alpha_1\alpha_2)$.
In the second part, we have
\begin{align*}
\mathcal{M}(\Phi,\alpha)\mathcal{M}(\Phi,\overline{\alpha})\mathcal{M}(\Phi,\alpha)\mathcal{M}(\Phi,\overline{\alpha}) 
\subseteq & \mathcal{M}(\Phi,\alpha)\mathcal{M}(\Phi,1) \mathcal{M}(\Phi,\overline{\alpha})\\
\subseteq & \mathcal{M}(\Phi,\alpha)\mathcal{M}(\Phi,\overline{\alpha}),
\end{align*}
and
\begin{align*}
[\mathcal{M}(\Phi,\alpha)\mathcal{M}(\Phi,\overline{\alpha})]^*=\mathcal{M}(\Phi,\overline{\alpha})^*\mathcal{M}(\Phi,\alpha)^*=\mathcal{M}(\Phi,\alpha)\mathcal{M}(\Phi,\overline{\alpha}),
\end{align*}
where $\mathcal{A}\mathcal{B}:=\text{span}\{ab:\ a\in\mathcal{A},\ b\in\mathcal{B}\}$ and $\mathcal{A}^*:=\{a^*:\ a\in\mathcal{A} \}$ for any $*$-algebra $\mathcal{A}$.

(iv) and (v) 
Let $x\in\mathcal{M}(\Phi,\alpha)$, $x\neq0$. 
By the polar decomposition $x=v|x|$, where $v\in\mathcal{M}$ is the partial isometry with the initial space $\mathcal{R}(x^*)$ and the final space $\mathcal{R}(x)$.
Since $x^*x\in\mathcal{M}(\Phi,\overline{\alpha})\mathcal{M}(\Phi,\alpha)\subseteq\mathcal{M}(\Phi,1)$ and $\mathcal{M}(\Phi,1)$ is a von Neumann algebra, we obtain $|x|\in\mathcal{M}(\Phi,1)$, i.e., $y_\Phi|x|=|x|y_\Phi$.
Thus  $y_\Phi\mathcal{R}(x^*)=\mathcal{R}(x^*)y_\Phi.$
Moreover, 
\begin{align*}
\alpha y_\Phi^* v|x|=\alpha y_\Phi^* x=xy_\Phi^*=v|x|y_\Phi^*=vy_\Phi^*|x|.
\end{align*}
Thus $(\alpha y_\Phi^* v- vy_\Phi^*)|x|=0$ and then $(\alpha y_\Phi^* v- vy_\Phi^*)\mathcal{R}(x^*)=0.$
We obtain that $vy_\Phi^*=\alpha y_\Phi^* v$.
Similarly, $y_\Phi v=\alpha vy_\Phi$.
Hence $v\in\mathcal{M}(\Phi,\alpha)$.
Note that $vv^*,v^*v\in\mathcal{M}(\Phi,1)$ and $\mathcal{M}(\Phi,1)$ is a finite factor. 
Then there exist finitely many  partial isometries $\{v_{j}\}_{j=1}^{\ell}$ and $  \{w_{j}\}_{j=1}^{\ell}$ in $\mathcal{M}(\Phi,1)$ such that $\displaystyle \sum_{j=1}^{\ell} v_{j} vw_{j}$ is a unitary.
We denote this unitary by $u_{\alpha}$, and then $u_{\alpha}\in\mathcal{M}(\Phi,\alpha)$.
Since 
\begin{align*}
u_{\alpha}\mathcal{M}(\Phi,1)\subseteq\mathcal{M}(\Phi,\alpha) \text{ and } u_{\alpha}^*\mathcal{M}(\Phi,\alpha)\subseteq\mathcal{M}(\Phi,1),
\end{align*}
 we have $\mathcal{M}(\Phi,\alpha)=u_{\alpha}\mathcal{M}(\Phi,1)$.
Since $u_{\alpha}\mathcal{M}(\Phi,1)u_{\alpha}^*=\mathcal{M}(\Phi,1)$, we have $u_{\alpha}\mathcal{M}(\Phi,1)=\mathcal{M}(\Phi,1)u_{\alpha}$.
Similarly, there exists a unitary $u_{\alpha^n}\in\mathcal{M}(\Phi,\alpha^n)$ such that $\mathcal{M}(\Phi,\alpha^n)=\mathcal{M}(\Phi,1)u_{\alpha^n}$.
Then 
\begin{align*}
    \mathcal{M}(\Phi,\alpha)^n=\mathcal{M}(\Phi,1)u_{\alpha}^n=(\mathcal{M}(\Phi,1)u_{\alpha}^n u_{\alpha^n}^* )u_{\alpha^n}=\mathcal{M}(\Phi,1)u_{\alpha^n}= \mathcal{M}(\Phi,\alpha^n).
\end{align*}
The third equality is due to the fact that $u_{\alpha}^n u_{\alpha^n}^*$ is a unitary in $\mathcal{M}(\Phi,1)$.
If $\alpha\in \Gamma$, then $\overline{\alpha}\in\Gamma$ from (ii).
If $\alpha,\beta\in\Gamma$, then there exist unitary $u_{\alpha}\in\mathcal{M}(\Phi,\alpha)$ and $u_{\beta}\in\mathcal{M}(\Phi,\beta)$. 
From (iii), $0\neq u_{\alpha}u_{\beta}\in\mathcal{M}(\Phi,\alpha\beta)$, and hence $\alpha\beta\in\Gamma$.
Therefore,  $\Gamma$ is an abelian subgroup of $ U(1)$.
Since $\sigma(\Phi)=\sigma(y_\Phi)$, $\Gamma$ is a finite group.
Hence $\Gamma$ is a finite cyclic group.
The proof is completed.
\end{proof}
\begin{remark}
In Theorem \ref{thm:entry}, we see that $\mathcal{M}(\Phi,1)=\mathcal{P}(\Phi,1)$ is a von Neumann subalgebra of $\mathcal{M}$.
The conditions on $\Phi$ can be relaxed to only require  that  $\Phi$ is a unital 2-positive map and there exists a faithful normal state $\omega$ on $\mathcal{M}$ such that $\omega\circ\Phi=\omega$ (see \cite[Theorem 2.2.11]{Sto13}).
However, Theorem \ref{thm:entry} is not true for unital 2-positive maps.
The finite index is crucial in the proof of Theorem \ref{thm:entry}.
\end{remark}

\begin{remark}
    For more details on bimodule categories, we refer to \cite{ Bis97,Con94, Pop86}.
\end{remark}

Suppose that $\mathcal{N}\subseteq\mathcal{M}$ is irreducible and $\Phi$ is a bimodule quantum channel.
By Proposition \ref{prop:PF theorem for bimodule map}, there exists $\Psi\in \CP$ such that $\Psi\Phi=r\Psi$,
which is equivalent to $\Phi^*\Psi^*=r\Psi^*$.
Note that $\mathcal{N}'\cap \mathcal{M}=\mathbb{C}\mathbf{1}$.
We have $\Psi^*(\mathbf{1})=c\mathbf{1}$, where $c>0$.
Hence $\Phi^*(\mathbf{1})=r\mathbf{1}$.
We can immediately obtain the following corollary.
\begin{corollary}
    Suppose $\mathcal{N}\subseteq\mathcal{M}$ is an irreducible subfactor and $\Phi$ is a bimodule quantum channel.
    Then all the statements (i) to (v) in Theorem \ref{thm:entry} hold for $\Phi$.
    In particular, $\mathcal{P}(\Phi,1)$ is an intermediate factor of $\mathcal{N}\subseteq\mathcal{M}$.
\end{corollary}

\begin{remark}
    We call the finite cyclic group $\Gamma$ in Theorem \ref{thm:entry} the phase group of the bimodule map $\Phi$.
\end{remark}

\begin{remark}
Let $(\mathcal{M},\phi, \xi)$ be an ergodic dynamic system defined in \cite{AlbHoe78} by Albeverio and H{\o}egh-Krohn.
In \cite[Theorem 2.2]{AlbHoe78}, they proved that $\Gamma(\phi)$,  the set of all the discrete eigenvalues on the unit circle for $\phi$,  is an abelian group.
If $\alpha\in\Gamma(\phi)$ then $\alpha$ is the simple eigenvalue of $\phi$ and $\phi(u_{\alpha})=\alpha u_{\alpha}$, where
$u_{\alpha}$ is a unitary operator in $\mathcal{M}$.
In Theorem \ref{thm:entry}, we prove that if $\alpha\in \Gamma$, then the eigenvector space $\mathcal{M}(\Phi,\alpha)$ is an $\mathcal{N}$-$\mathcal{N}$-bimodule.
In particular, $\mathcal{M}(\Phi,1)$ is an intermediate von Neumann algebra of $\mathcal{N}\subseteq\mathcal{M}$.
\end{remark}

\begin{definition}[Relative Irreducibility]
    Suppose $\mathcal{N}\subseteq\mathcal{M}$ is an inclusion of von Neumann algebras.
    We say an $\mathcal{N}$-$\mathcal{N}$-bimodule positive map $\Phi$: $\mathcal{M}\to\mathcal{M}$ is {\sl relatively irreducible} if 
 $\Phi(p)\leq cp$ for some projection $p\in\mathcal{M}$ and positive number $c>0$ implies that $p\in\mathcal{N}$.
\end{definition}

Suppose $\mathcal{N}\subseteq\mathcal{M}$ is a subfactor of finite index.
Let $\Phi\in \mathbf{CP}_{\mathcal{N}}(\mathcal{M})$.
Recall that $\Phi^*$ is defined in Remark \ref{rem:operation bimodule map} such that $\tau(\Phi(x)y)=\tau(x\Phi^*(y))$ for any $x,y\in\mathcal{M}$.
One can check that $\Phi$ is relatively irreducible if and only if $\Phi^*$ is relatively irreducible.

\begin{remark}\label{rem:relative irr and irr}
    Suppose $\mathcal{N}\subseteq\mathcal{M}$ is an inclusion of von Neumann algebras.
    If $\mathcal{N}=\mathbb{C}\mathbf{1}$ then the relative irreducibility is the irreducibility in \cite{EvaHoe78}.
   If $\mathcal{N}$ is a factor then the relative irreducibility of a bimodule positive map $\Phi$ implies the irreducibility of $\Phi|_{\mathcal{N}'\cap\mathcal{M}}$. 
\end{remark}

\begin{lemma}
     Suppose $\mathcal{N}\subseteq\mathcal{M}$ is an inclusion of von Neumann algebras and $\Phi$ is a relatively irreducible bimodule quantum channel on $\mathcal{M}$.
     Then $\Phi$ is faithful.
\end{lemma}
\begin{proof}
    Suppose that $\Phi(x)=0$ for some $x\in\mathcal{M}^+$.
    By the Kadison-Schwarz inequality, $\Phi(x^{1/2})^2\leq\Phi(x)=0$, and hence $\Phi(x^{1/2})=0$.
    Continuing this step, we obtain $\Phi(x^{1/2^n})=0$ for any $n\geq1$.
    Therefore, $\Phi(\mathcal{R}(x))=\displaystyle \lim_{n\to \infty}\Phi(x^{1/2^n})=0$ and $\Phi(\mathcal{R}(x))\leq \mathcal{R}(x)$.
Hence $\mathcal{R}(x)\in\mathcal{N}$ by the  relative irreducibility of $\Phi$, and then $\mathcal{R}(x)=\Phi(\mathcal{R}(x))=0$.
    Hence $\Phi$ is faithful.
\end{proof}

\begin{lemma}\label{lem:faithful state}
   Suppose $\mathcal{N}\subseteq\mathcal{M}$ is a subfactor of finite index and $\Phi$ is a relatively irreducible bimodule quantum channel on $\mathcal{M}$.
   Then there exists a normal faithful state $\omega$ on $\mathcal{M}$ such that $\omega\Phi=\omega$.
\end{lemma}
\begin{proof}
Since $\Phi$ is relatively irreducible, $\Phi^*$ is also relatively irreducible.
By Remark \ref{rem:relative irr and irr}, $\Phi^*|_{\mathcal{N}'\cap\mathcal{M}}$ is irreducible.
Then there exists a strictly positive element $x\in \mathcal{N}'\cap\mathcal{M}$ such that  $\Phi^*(x)=x$ and $\tau(x)=1$.
We define $\omega (y):=\tau(yx)$, $y\in\mathcal{M}$.
Then $\omega$ is a normal faithful  state on $\mathcal{M}$.
Moreover, $\omega(\Phi(y))=\tau(\Phi(y)x)=\tau(y\Phi^*(x))=\tau(yx)=\omega(y)$, $y\in\mathcal{M}$.
Hence $\omega\Phi=\omega$.
    \end{proof}

\begin{theorem}[Relative Irreducibility]\label{thm:RI Frobenius factor}
    Suppose $\mathcal{N}\subseteq\mathcal{M}$ is a subfactor of finite index and $\Phi$ is a relatively irreducible bimodule quantum channel.
  Then the eigenvalues of $\Phi$ with modulus $1$ form a finite cyclic subgroup $\Gamma$ of the unit circle $ U(1)$.
  The  space of fixed points $\mathcal{M}(\Phi,1)=\mathcal{N}$.
  For each $\alpha\in\Gamma$, there exists a unitary $u_{\alpha}\in\mathcal{M}(\Phi,\alpha)$ such that $\mathcal{M}(\Phi,\alpha)=u_{\alpha}\mathcal{N}=\mathcal{N}u_{\alpha}$.
  Moreover, $\mathcal{M}(\Phi, \alpha)$ are invertible $\mathcal{N}$-$\mathcal{N}$-bimodules and they form a unitary fusion category.
\end{theorem}
\begin{proof}
  By Lemma \ref{lem:faithful state}, there exists a faithful normal state $\omega$ on $\mathcal{M}$ such that $\omega\Phi=\omega$.
  For any projection $p\in\mathcal{M}(\Phi,1)$, $\Phi(p)=p$.
    Hence $p\in\mathcal{N}$ by the relative irreducibility, and then $\mathcal{M}(\Phi,1)=\mathcal{N}$.
The other statements directly follow from Theorem \ref{thm:entry}.
\end{proof}

\section{Phase groups for $\mathfrak{F}$-positive elements}\label{sec:Phase groups for F positive elements}
In this section, we show an intrinsic description of the phase groups of bimodule maps in terms of planar algebras.
Suppose $\mathscr{P}=\{\mathscr{P}_{k,\pm}\}_{k\geq 0}$ is an irreducible subfactor planar algebra. 
For any $x\in\mathscr{P}_{2,+}$, its Fourier transform $\mathfrak{F}(x)$ is defined in Equation \eqref{eq:Fourier transform}.
For simplicity, we denote it by $\widehat{x}$.
Recall that an element $x\in\mathscr{P}_{2,+}$ is called $\mathfrak{F}$-positive if $\widehat{x}$ is positive (see \cite[Definition 2.6]{HJLW23}).
A projection $p\in\mathscr{P}_{2,\pm}$ is called a biprojection if $\widehat{p}$ is a positive multiple of a projection.
A projection $q\in\mathscr{P}_{2,\pm}$ is said to be a left shift of a biprojection $p$ if $tr_2(p)=tr_2(q)$ and $\displaystyle q\ast p=\frac{tr_2(p)}{\sqrt{\mu}}q$; a projection $q\in\mathscr{P}_{2,\pm}$ is said to be a right shift of a biprojection $p$ if $tr_2(p)=tr_2(q)$ and $\displaystyle p\ast q=\frac{tr_2(p)}{\sqrt{\mu}}q$ (see \cite[Definition 6.5]{JLW16}).
We present some results in quantum Fourier analysis.
\begin{proposition}[Theorem 4.1, \cite{JLW19}]\label{prop:Sum set estimate}
    Suppose $\mathscr{P}$ is an irreducible subfactor planar algebra. 
    Let $p$ and $q$ be projections in $\mathscr{P}_{2,\pm}$.
    Then 
    \begin{align*}
        \max\{ tr_2(p),tr_2(q)\}\leq \mathcal{S}(p\ast q).
    \end{align*}
    Moreover, $\mathcal{S}(p\ast q)=tr_2(q)$ if and only if $\displaystyle \frac{\delta}{tr_2(p)}p\ast q$ is a projection.
\end{proposition}
\begin{proposition}[Proposition 3.2, \cite{JLW19}]\label{prop:biprojection}
     Suppose $\mathscr{P}$ is an irreducible subfactor planar algebra.
     Let $p\in\mathscr{P}_{2,\pm}$ be a projection such that $p$ is $\mathfrak{F}$-positive.
     Then $p$ is a biprojection.
\end{proposition}

\begin{lemma}\label{lem:split}
Suppose $\mathscr{P}$ is an irreducible subfactor planar algebra.
Let $x\in\mathscr{P}_{2,+}$ be an $\mathfrak{F}$-positive element with $tr_{2,-}(\widehat{x})=\mu^{1/2}$.
Then for any $y\in \mathscr{P}_{2,-}$ such that $y*\widehat{x}=\alpha y$ for some $|\alpha|=1$, we have
\begin{align}\label{eq:split1}
\raisebox{-0.8cm}{
\begin{tikzpicture}[xscale=2.2, yscale=1.2]
\draw [blue] (0,0) rectangle (0.5, 0.5);
\node at (0.25, 0.25) {$y$};
\draw (0.15, 0.5)--(0.15, 0.9) (0.15, 0)--(0.15, -0.4) (0.35, 0)--(0.35, -0.4);
\begin{scope}[shift={(0.7, 0)}]
\draw [blue] (0,0) rectangle (0.5, 0.5);
\node at (0.25, 0.25) {$\widehat{x}^{1/2}$};
\draw (0.35, 0.5)--(0.35, 0.9) (0.15, 0)--(0.15, -0.4) (0.35, 0)--(0.35, -0.4);
\end{scope}
\draw (0.35, 0.5) .. controls +(0, 0.3) and +(0, 0.3) .. (0.85, 0.5);
\end{tikzpicture}}
=\alpha
\raisebox{-1.1cm}{
\begin{tikzpicture}[xscale=2.2, yscale=1.2]
\draw [blue] (0,0) rectangle (0.5, 0.5);
\node at (0.25, 0.25) {$y$};
\draw   (0.15, 0.5)--(0.15, 0.8) (0.35, 0.5)--(0.35, 0.8) ;
\begin{scope}[shift={(0.3, -0.9)}]
\draw [blue] (0,0) rectangle (0.5, 0.5);
\node at (0.25, 0.25) {$\widehat{x}^{1/2}$};
\draw (0.15, 0)--(0.15, -0.3) (0.35, 0)--(0.35, -0.3);
\draw (0.15, 0.5) .. controls +(0, 0.25) and +(0, 0.25) .. (-0.25, 0.5)--(-0.25, -0.3);
\end{scope}
\draw[rounded corners] (0.35, 0)--(0.65, -0.4) (0.15, 0)--(-0.15, -0.4) -- (-0.15, -1.2) ;
\end{tikzpicture}}.
\end{align}
In particular, we have for any $z\in \mathscr{P}_{2,-}$, 
\begin{align}\label{eq:split2}
(zy)*\widehat{x}=\alpha (z*\widehat{x})y.
\end{align}
\end{lemma}
\begin{proof}
Let 
\begin{align*}
T=\raisebox{-0.8cm}{
\begin{tikzpicture}[xscale=2.2, yscale=1.2]
\draw [blue] (0,0) rectangle (0.5, 0.5);
\node at (0.25, 0.25) {$y$};
\draw (0.15, 0.5)--(0.15, 0.9) (0.15, 0)--(0.15, -0.4) (0.35, 0)--(0.35, -0.4);
\begin{scope}[shift={(0.7, 0)}]
\draw [blue] (0,0) rectangle (0.5, 0.5);
\node at (0.25, 0.25) {$\widehat{x}^{1/2}$};
\draw (0.35, 0.5)--(0.35, 0.9) (0.15, 0)--(0.15, -0.4) (0.35, 0)--(0.35, -0.4);
\end{scope}
\draw (0.35, 0.5) .. controls +(0, 0.3) and +(0, 0.3) .. (0.85, 0.5);
\end{tikzpicture}}
-\alpha
\raisebox{-1.1cm}{
\begin{tikzpicture}[xscale=2.2, yscale=1.2]
\draw [blue] (0,0) rectangle (0.5, 0.5);
\node at (0.25, 0.25) {$y$};
\draw   (0.15, 0.5)--(0.15, 0.8) (0.35, 0.5)--(0.35, 0.8) ;
\begin{scope}[shift={(0.3, -0.9)}]
\draw [blue] (0,0) rectangle (0.5, 0.5);
\node at (0.25, 0.25) {$\widehat{x}^{1/2}$};
\draw (0.15, 0)--(0.15, -0.3) (0.35, 0)--(0.35, -0.3);
\draw (0.15, 0.5) .. controls +(0, 0.25) and +(0, 0.25) .. (-0.25, 0.5)--(-0.25, -0.3);
\end{scope}
\draw[rounded corners] (0.35, 0)--(0.65, -0.4) (0.15, 0)--(-0.15, -0.4) -- (-0.15, -1.2) ;
\end{tikzpicture}}.
\end{align*}
Note that 
\begin{align*}    
    \raisebox{-0.55cm}{
\begin{tikzpicture}[scale=1.2]
\draw [blue] (0,0) rectangle (0.7, 0.5);
\node at (0.35, 0.25) {$\widehat{x}$};
\draw (-0.15, 0)--(-0.15, 0.5) (0.55, 0)--(0.55, -0.2) (0.55, 0.5)--(0.55, 0.7);
\draw (0.15, 0.5) .. controls +(0, 0.3) and +(0, 0.3) .. (-0.15, 0.5);
\draw (0.15, 0) .. controls +(0, -0.3) and +(0, -0.3) .. (-0.15, 0);
\end{tikzpicture}}= 1
\quad \text{and}\quad 
 \raisebox{-0.55cm}{
\begin{tikzpicture}[scale=1.2]
\path  (0.55, 0.5) .. controls +(0, 0.3) and +(0, 0.3) .. (0.85, 0.5) -- (0.85, 0) .. controls +(0, -0.3) and +(0, -0.3) ..  (0.55, 0) ;
\path  (-0.15, -0.3) rectangle (0.15, 0.8);
\draw (0.15, 0.8)--(0.15, -0.3);
\draw [blue, fill=white] (0,0) rectangle (0.7, 0.5);
\node at (0.35, 0.25) {$\widehat{x}$};
\draw (0.55, 0.5) .. controls +(0, 0.3) and +(0, 0.3) .. (0.85, 0.5) -- (0.85, 0) .. controls +(0, -0.3) and +(0, -0.3) ..  (0.55, 0) ;
\end{tikzpicture}}=1\;.
\end{align*}
We have
\begin{align*}
tr_{2,-}(T^*T)=& 
\raisebox{-0.6cm}{
\begin{tikzpicture}[xscale=1.8, yscale=1.2]
\draw [blue] (0,0) rectangle (0.5, 0.5);
\node at (0.25, 0.25) {$y^*y$};
\draw (0.15, 0.5) .. controls +(0, 0.25) and +(0, 0.25) .. (-0.15, 0.5) -- (-0.15, 0) .. controls +(0, -0.25) and +(0, -0.25) ..  (0.15, 0);
\begin{scope}[shift={(0.7, 0)}]
\draw [blue] (0,0) rectangle (0.5, 0.5);
\node at (0.25, 0.25) {$\widehat{x}$};
\draw (0.35, 0.5) .. controls +(0, 0.25) and +(0, 0.25) .. (0.65, 0.5) -- (0.65, 0) .. controls +(0, -0.25) and +(0, -0.25) ..  (0.35, 0);
\end{scope}
\draw (0.35, 0.5) .. controls +(0, 0.3) and +(0, 0.3) .. (0.85, 0.5);
\draw (0.35, 0) .. controls +(0, -0.3) and +(0, -0.3) .. (0.85, 0);
\end{tikzpicture}}
+
\raisebox{-1.8cm}{
\begin{tikzpicture}[xscale=1.8, yscale=1.2]
\draw [blue] (0,0) rectangle (0.5, 0.5);
\node at (0.25, 0.25) {$y$};
\draw (0.15, 0.5)  .. controls +(0, 0.4) and +(0, 0.4) .. (-0.3, 0.5)--(-0.3, -1.8)  .. controls +(0, -0.4) and +(0, -0.4) .. (0.15, -1.8) ; 
\draw (0.35, 0.5)  .. controls +(0, 0.4) and +(0, 0.4) .. (0.9, 0.5)--(0.9, -1.8)  .. controls +(0, -0.4) and +(0, -0.4) .. (0.35, -1.8) ; 
\begin{scope}[shift={(0.3, -0.9)}]
\draw [blue] (0,0) rectangle (0.5, 0.5);
\node at (0.25, 0.25) {$\widehat{x}$};
\draw (0.15, 0.5) .. controls +(0, 0.25) and +(0, 0.25) .. (-0.25, 0.5) -- (-0.25, 0) .. controls +(0, -0.25) and +(0, -0.25) .. (0.15, 0);
\end{scope}
\draw[rounded corners] (0.35, 0)--(0.65, -0.4) (0.35, -1.3)--(0.65, -0.9) (0.15, 0)--(-0.15, -0.4) -- (-0.15, -0.9)--(0.15, -1.3);
\begin{scope}[shift={(0, -1.8)}]
\draw [blue] (0,0) rectangle (0.5, 0.5);
\node at (0.25, 0.25) {$y^*$};
\end{scope}
\end{tikzpicture}}
-\alpha
\raisebox{-1.1cm}{
\begin{tikzpicture}[xscale=1.8, yscale=1.2]
\draw [blue] (0,0) rectangle (0.5, 0.5);
\node at (0.25, 0.25) {$y$};
\draw (0.15, 0.5)  .. controls +(0, 0.4) and +(0, 0.4) .. (-0.4, 0.5)--(-0.4, -0.9)  .. controls +(0, -0.2) and +(0, -0.2) .. (-0.15, -0.9) ; 
\draw (0.35, 0.5)  .. controls +(0, 0.4) and +(0, 0.4) .. (0.9, 0.5)--(0.9, -0.9)  .. controls +(0, -0.2) and +(0, -0.2) .. (0.65, -0.9) ; 
\begin{scope}[shift={(0.3, -0.9)}]
\draw [blue] (0,0) rectangle (0.5, 0.5);
\node at (0.25, 0.25) {$\widehat{x}$};
\draw (0.15, 0.5) .. controls +(0, 0.25) and +(0, 0.25) .. (-0.25, 0.5);
\draw (0.15, 0) .. controls +(0, -0.25) and +(0, -0.25) .. (-0.25, 0);
\end{scope}
\begin{scope}[shift={(-0.3, -0.9)}]
\draw [blue] (0,0) rectangle (0.5, 0.5);
\node at (0.25, 0.25) {$y^*$};
\end{scope}
\draw[rounded corners] (0.35, 0)--(0.65, -0.4) (0.15, 0)--(-0.15, -0.4) ;
\end{tikzpicture}}
-\overline{\alpha}
\raisebox{-1.1cm}{
\begin{tikzpicture}[xscale=1.8, yscale=1.2]
\begin{scope}[shift={(0, -1.8)}]
\draw [blue] (0,0) rectangle (0.5, 0.5);
\node at (0.25, 0.25) {$y^*$};
\draw[rounded corners] (0.35, 0.5)--(0.65, 0.9) (0.15, 0.5)--(-0.15, 0.9) ;
\end{scope}
\draw (0.15, -1.8)  .. controls +(0, -0.4) and +(0, -0.4) .. (-0.4, -1.8)--(-0.4, -0.4)  .. controls +(0, 0.2) and +(0, 0.2) .. (-0.15, -0.4) ; 
\draw (0.35, -1.8)  .. controls +(0, -0.4) and +(0, -0.4) .. (0.9, -1.8)--(0.9, -0.4)  .. controls +(0, 0.2) and +(0, 0.2) .. (0.65, -0.4) ; 
\begin{scope}[shift={(0.3, -0.9)}]
\draw [blue] (0,0) rectangle (0.5, 0.5);
\node at (0.25, 0.25) {$\widehat{x}$};
\draw (0.15, 0.5) .. controls +(0, 0.25) and +(0, 0.25) .. (-0.25, 0.5);
\draw (0.15, 0) .. controls +(0, -0.25) and +(0, -0.25) .. (-0.25, 0);
\end{scope}
\begin{scope}[shift={(-0.3, -0.9)}]
\draw [blue] (0,0) rectangle (0.5, 0.5);
\node at (0.25, 0.25) {$y$};
\end{scope}
\end{tikzpicture}} \\
= & 
tr_{2,-}(y^*y)
+ tr_{2,-}(y^*y)-tr_{2,-}(y^*y)-tr_{2,-}(y^*y)\\
= & 0.
\end{align*}
We then obtain that $T=0$.
Multiplying the following element to Equation (\ref{eq:split1})
\begin{align*}
\raisebox{-0.8cm}{
\begin{tikzpicture}[xscale=2.2, yscale=1.2]
\draw [blue] (0,0) rectangle (0.5, 0.5);
\node at (0.25, 0.25) {$z$};
\draw (0.15, 0)--(0.15, -0.4) (0.15, 0.5)--(0.15, 0.9) (0.35, 0.5)--(0.35, 0.9);
\begin{scope}[shift={(0.7, 0)}]
\draw [blue] (0,0) rectangle (0.5, 0.5);
\node at (0.25, 0.25) {$\widehat{x}^{1/2}$};
\draw (0.35, 0)--(0.35, -0.4) (0.15, 0.5)--(0.15, 0.9) (0.35, 0.5)--(0.35, 0.9);
\end{scope}
\draw (0.35, 0) .. controls +(0, -0.3) and +(0, -0.3) .. (0.85, 0);
\end{tikzpicture}}
\end{align*}
we obtain Equation (\ref{eq:split2}).
This completes the proof of the lemma.
\end{proof}

\begin{theorem}\label{thm:Frobenius irreducible}
Suppose $\mathscr{P}$ is an irreducible subfactor planar algebra.
Let $x\in\mathscr{P}_{2,+}$ be an $\mathfrak{F}$-positive element  with the spectral radius $r=1$. 
Suppose that $\lambda_j\in\sigma(x)$  with $|\lambda_j|=1$ and $q_j\in\mathscr{P}_{2,+}$ is the spectral projection corresponding to $\lambda_j$ for $1\leq j\leq m$.
Assume that $\lambda_1=1$.
Then we have the following statements.
\begin{enumerate}[(i)]
\item $\|x\|_{\infty}=tr_{2,-}(\widehat{x})/\mu^{1/2}=1$.
 \item The set $\Gamma:=\{\lambda_j\}_{j=1}^m$ forms a finite cyclic group of order $m$ under the multiplication of complex numbers.
    \item $q_1\in\mathscr{P}_{2,+}$ is a biprojection. Moreover,
    \begin{align*}
q_1=\lim_{n\to\infty}\frac{1}{n}\sum_{k=1}^{n}x^k.
    \end{align*}
    \item $q_j$ is a right shift of the biprojection $q_1$.
    \item $\displaystyle \sum_{j=1}^m q_j$ is a biprojection.
    Moreover, 
    \begin{align*}
        \displaystyle \sum_{j=1}^m q_j=\lim_{n\to\infty}\frac{1}{n}\sum_{k=1}^n(xx^*)^k.
    \end{align*}
    \item $\{q_{j}\}_{j=1}^m$ forms a finite cyclic group under convolution up to a fixed scalar.
\end{enumerate}
\end{theorem}
\begin{proof}
(i) Since $\mathscr{P}$ is irreducible,
\begin{align}\label{eq:irreducible element}
\raisebox{-0.55cm}{
\begin{tikzpicture}[rotate=180,scale=1.2]
\path  (0.55, 0.5) .. controls +(0, 0.3) and +(0, 0.3) .. (0.85, 0.5) -- (0.85, 0) .. controls +(0, -0.3) and +(0, -0.3) ..  (0.55, 0) ;
\path  (-0.15, -0.3) rectangle (0.15, 0.8);
\draw (0.15, 0.8)--(0.15, -0.3);
\draw [blue, fill=white] (0,0) rectangle (0.7, 0.5);
\node at (0.35, 0.25) {$\widehat{x}$};
\draw (0.55, 0.5) .. controls +(0, 0.3) and +(0, 0.3) .. (0.85, 0.5) -- (0.85, 0) .. controls +(0, -0.3) and +(0, -0.3) ..  (0.55, 0) ;
\end{tikzpicture}}=
\frac{tr_{2,-}(\widehat{x})}{\mu^{1/2}}
\raisebox{-0.55cm}{
\begin{tikzpicture}[scale=1.2]
\draw (0.15, 0.8)--(0.15, -0.3);
\end{tikzpicture}}\;.
\end{align}
Thus, $tr_{2,-}(\widehat{x})/\mu^{1/2}\leq 1$.
On the other hand, the Hausdorff-Young inequality \cite[Theorem 4.8]{JLW16} assures that $$1\leq\|x\|_{\infty}=\|\mathfrak{F}^{-1}(\widehat{x})\|_{\infty}\leq \|\widehat{x}\|_1/\mu^{1/2}=tr_{2,-}(\widehat{x})/\mu^{1/2}.$$
Therefore,
$$\|x\|_{\infty}=tr_{2,-}(\widehat{x})/\mu^{1/2}=1.$$
(ii) Since $r=\|x\|_{\infty}=1$, we have $q_jx=xq_j=\lambda_j q_j$. 
Applying Lemma \ref{lem:split} to $\mathfrak{F}^{-2}(x)$ and taking $y=\mathfrak{F}^{-1}(q_j)$ and $z=\mathfrak{F}^{-1}(q_k)$, we have
  \begin{align*}
    [\mathfrak{F}^{-1}(q_k)  \mathfrak{F}^{-1}(q_j) ]\ast \mathfrak{F}^{-1}(x)=\lambda_j\lambda_k  \mathfrak{F}^{-1}(q_k)  \mathfrak{F}^{-1}(q_j).   
  \end{align*}
Taking the Fourier transform, we have
\begin{align*}
x (q_k\ast q_j)=\lambda_k\lambda_j (q_k\ast q_j).
\end{align*}
  Since $q_k\ast q_j\neq0$, then $\lambda_k\lambda_j\in\Gamma$.
  On the other hand,
  \begin{align*}
xJq_jJ=\overline{\lambda_j}Jq_jJ,
  \end{align*}
  which implies $\lambda_j^{-1}=\overline{\lambda_j}\in\Gamma$.
  Therefore, $\Gamma$ is a finite subgroup of the unit circle. 
  Then we see that $\Gamma$ is a finite cyclic group.

(iii) Let $\displaystyle x=\sum_{j=1}^m\lambda_j q_j+x'$ such that $q_jx'=x'q_j=0$ and $r(x')<1$.
From (ii), we assume $\lambda_j={\rm e}^{2\pi i(j-1)/m}$, $j=1,\ldots,m$.
It follows that
\begin{align*}
\lim_{n\to\infty}\frac{1}{n}\sum_{k=1}^{n}x^k=\lim_{n\to\infty}\frac{1}{n}\sum_{k=1}^{n}\sum_{j=1}^m\lambda_j^kq_j+\lim_{n\to\infty}\frac{1}{n}\sum_{k=1}^{n}x'^k=q_1.
\end{align*}
   Moreover, the Fourier transform of $q_1$,
\begin{align*}
\widehat{q_1}=\lim_{n\to\infty}\frac{1}{n}\sum_{k=1}^{n}\widehat{x}^{\ast k}\geq0,
\end{align*}
by the quantum Schur product theorem \cite[Theorem 4.1]{Liu16}.
  Consequently, $q_1$ is a biprojection by Proposition \ref{prop:biprojection}.

  (iv)  From (ii), we have
\begin{align*}
    x (q_1\ast q_j)=\lambda_1\lambda_j (q_1\ast q_j)=\lambda_j (q_1\ast q_j).
\end{align*}
Therefore, $\mathcal{R}(q_1\ast q_j)\leq q_j$ and $tr_{2,+}(\mathcal{R}(q_1\ast q_j))\leq tr_{2,+}(q_j)$.
By Proposition \ref{prop:Sum set estimate}, we have 
\begin{align*}
tr_{2,+}(q_1)\leq tr_{2,+}(\mathcal{R}(q_1\ast q_j))= tr_{2,+}(q_j)
\end{align*}
and $\displaystyle \frac{\mu^{1/2}}{tr_{2,+}(q_1)}q_1\ast q_j$ is a projection.
Hence $\displaystyle \frac{\mu^{1/2}}{tr_{2,+}(q_1)}q_1\ast q_j=q_j$.
Let $q_{-j}$ be the spectral projection corresponding to $\lambda_j^{-1}$.
Then 
\begin{align*}
    x (q_{-j}\ast q_j)= q_{-j}\ast q_j.
\end{align*}
Therefore, $\mathcal{R}(q_{-j}\ast q_j)\leq q_1$ and $tr_{2,+}(\mathcal{R}(q_{-j}\ast q_j))\leq tr_{2,+}(q_1)$.
Again by Proposition \ref{prop:Sum set estimate}, we have $tr_{2,+}(q_j)\leq tr_{2,+}(q_1)$.
Thus, $tr_{2,+}(q_j)= tr_{2,+}(q_1)$.
We obtain that $q_j$ is a right shift of $q_1$ for  $1\leq j \leq m$.

(v) For any $1\leq j\leq m$, we have
\begin{align*}
    q_j=\lim_{n\to\infty}\frac{1}{n}\sum_{k=1}^{n}(\lambda_j^{-1}x)^k.
\end{align*}
Then
\begin{align*}
\sum_{j=1}^m q_j=\lim_{n\to\infty}\frac{1}{n}\sum_{k=1}^{n}\sum_{j=1}^m (\lambda_j^{-1}x)^k=
\lim_{n\to\infty}\frac{m}{n}\sum_{k=1}^{n}x^{km}.
\end{align*}
Therefore, the projection $\displaystyle \sum_{j=1}^m q_j$ is $\mathfrak{F}$-positive. 
By Proposition \ref{prop:biprojection}, we obtain that $\displaystyle \sum_{j=1}^m q_j$ is a biprojection.
Moreover, by the above discussion,
\begin{align*}
    \lim_{n\to\infty}\frac{1}{n}\sum_{k=1}^n(xx^*)^k= \sum_{j=1}^m q_j+  \lim_{n\to\infty}\frac{1}{n}\sum_{k=1}^n(x'x'^*)^k= \sum_{j=1}^m q_j.
\end{align*}
(vi) By the above discussions, one can check that $\displaystyle \frac{\mu^{1/2}}{tr_{2,+}(q_1)}q_k\ast q_j=q_{k+j-1}$.
The proof is completed.
\end{proof}

\begin{remark}
    Suppose $\mathcal{N}\subseteq \mathcal{M}$ is an irreducible type II$_1$ subfactor and $\Phi\in \CP$ such that $\widehat{\Phi}$ is a trace-one projection.
    Then $y_{\Phi}$ is a multiple of a unitary element in $\mathcal{N}'\cap \mathcal{M}$.
    By Theorem \ref{thm:Frobenius irreducible}, there exists a biprojection $q_1\in \mathcal{N}'\cap \mathcal{M}$ and right shifts $q_2, \ldots, q_m$ of the biprojection $q_1$ such that $\displaystyle y_{\Phi}=\mu^{-1}\sum_{j=1}^m \lambda_0^{j-1} q_j$, where $\lambda_0$ is an $m$-th root of unity.
    The phase group $\Gamma$ of $\Phi$ is inherited in $y_{\Phi}$.
    The unitary elements obtained in the statement $(v)$ of Theorem \ref{thm:entry} corresponding to $\lambda_0$ produce the characters of the phase group $\Gamma$ by considering its spectral decomposition.
    This implies Proposition 1.7 in \cite{PimPop86}.
    In this sense, Theorem \ref{thm:entry} and Theorem \ref{thm:Frobenius irreducible} extend it dramatically.
\end{remark}

\begin{remark}
Suppose $\mathcal{N}\subseteq \mathcal{M}$ is an irreducible type II$_1$ subfactor and $\Phi\in \CP$ is a bimodule quantum channel.
    Let $\mathcal{P}$ be the eigenspace of $\Phi$ corresponding to the eigenvalue $1$. We see that $\mathcal{P}$ is an intermediate subfactor and $\Phi$ is relatively irreducible as a $\mathcal{P}$-$\mathcal{P}$-bimodule quantum channel.
    Suppose that $p\in \mathcal{M}$ is a projection such that $\Phi(p)\leq cp$ for some $c>0$.
    By the fact that $\displaystyle E_{\mathcal{P}}=\lim_{n\to\infty} \frac{1}{n}\sum_{k=1}^n\Phi^k$, we see that $E_{\mathcal{P}}(p)\leq c_1 p$ for some $c_1>0$.
    Then by the Pimsner-Popa inequality, we see that $\mathcal{R}(E_{\mathcal{P}}(p))=p$, i.e., $p\in \mathcal{P}$.
    This implies that $\Phi$ is relatively irreducible as a $\mathcal{P}$-$\mathcal{P}$-bimodule quantum channel.
\end{remark}

We recall from \cite{Liu16} that given a positive element $y\in\mathscr{P}_{2,-}$ the biprojection $q$ generated by $y$ is the range projection of $\displaystyle \sum_{i=1}^ky^{\ast i}$ for $k$ large enough.
Let $p$ be the biprojection generated by $y\ast \overline{y}$.
From Theorem \ref{thm:Frobenius irreducible}, we can obtain the precise relation between $p$ and $q$ as follows.
\begin{theorem}\label{thm:two biprojection}
Suppose $\mathscr{P}$ is an irreducible subfactor planar algebra and  $y\in\mathscr{P}_{2,+}$ is a positive element.
Let $q$ and $p$ be the biprojections generated by $y$ and $y\ast \overline{y}$ respectively.
 Then $\displaystyle q=\sum_{j=1}^m p_{j}$, where $p_{j}$ is a right shift of the biprojection $p$.
  Moreover $\{p_j\}_{j=1}^m$ forms a finite cyclic group under convolution.
\end{theorem}
\begin{proof}
    Let $x\in\mathscr{P}_{2,+}$ be an $\mathfrak{F}$-positive element such that $\widehat{x}=y$.
    Now let $q$ and $p$ be the range projections of $\widehat{q}_1$ and $\displaystyle \sum_{j=1}^m\widehat{q_j}$ respectively.
    We define $\displaystyle p_j=\dfrac{\mu^{1/2}}{tr_2(q_1)m}\sum_{k=1}^m {\rm e}^{\frac{2\pi ij(k-1)}{m}}\widehat{q_k}$, $j=1,\ldots,m$.
    One can check that the results hold by Theorem \ref{thm:Frobenius irreducible}.
\end{proof}

\section{$\lambda$-extension}\label{sec:extension}
In this section, we will generalize the main results in \S \ref{sec:PF theory for CPB} to inclusion of von Neumann algebras in terms of $\lambda$-extension.
Let $\mathcal{N}\subseteq\mathcal{M}$ be an inclusion of finite von Neumann algebras.
Let $\tau_{\mathcal{M}}(=\tau)$ be a fixed normal faithful tracial state on $\mathcal{M}$.
Let $\mathcal{M}_1=\langle \mathcal{M}, e_1\rangle$ be the von Neumann algebra acting on $L^2(\mathcal{M})$ generated by $\mathcal{M}$ and $e_1$.
The inclusion $\mathcal{N}\subseteq \mathcal{M}$ is called finite-index if $\mathcal{M}_1$ is a finite von Neumann algebra.

We always assume that $\mathcal{M}$ is a $\lambda$-extension of $\mathcal{N}$, i.e., there exists a normal faithful tracial state $\tau_{1}$ on $\mathcal{M}_1$ such that $\tau_1|_{\mathcal{M}}=\tau$ and $E_{\mathcal{M}}(e_1)=\lambda \mathbf{1}$ for some positive constant $\lambda$, where $E_{\mathcal{M}}$ is the $\tau_1$-preserving conditional expectation onto $\mathcal{M}$.
In this case, there exists a Pimsner-Popa basis $\displaystyle \{\eta_j\}_{j=1}^m$ of $\mathcal{N}\subseteq \mathcal{M}$, namely
\begin{align*}
x=\sum_{j=1}^m  E_{\mathcal{N}}(x\eta_j^*)\eta_j, \quad \text{for all } x\in\mathcal{M}.
\end{align*}
Equivalently, $\displaystyle x=\sum_{j=1}^m \eta_j^* E_{\mathcal{N}}(\eta_j x)$.
We also assume that $\mathcal{M}_2$ is a $\lambda$-extension of $\mathcal{M}_1$.

The index of $\mathcal{N}\subseteq\mathcal{M}$ is $\tau_1(e_1)^{-1}= \displaystyle \sum_{j=1}^m\tau(\eta_j\eta_j^*)=\lambda^{-1}$.
The Pimsner-Popa inequality reads
\begin{align}\label{eq:PP inequality finite vN algbera}
    E_{\mathcal{N}}(x)\geq \lambda_{\mathcal{N}\subseteq \mathcal{M}} x,\quad\text{ for all } x\in\mathcal{M}^+,
\end{align}
where $\lambda_{\mathcal{N}\subseteq \mathcal{M}}$ is the Pimsner-Popa constant.
The inclusion $\mathcal{N}\subseteq\mathcal{M}$ is finite if and only if $\lambda_{\mathcal{N}\subseteq \mathcal{M}}>0$.
When $\mathcal{N}\subseteq \mathcal{M}$ is an inclusion of type II$_1$ factors, the Pimsner-Popa constant is $\lambda$.


By \cite{Bis97}, we shall identify $\CB$ with $\mathcal{N}'\cap \mathcal{M}_1$ and the bimodule $_{\mathcal{N}}\mathcal{M}_{\mathcal{N}}$.
For each $\Phi\in \CB$, the corresponding element $y_{\Phi}$ is defined as 
\begin{align*}
y_{\Phi} x\Omega=\Phi(x)\Omega \text{ for all } x\in \mathcal{M}.
\end{align*}
We denote by $\widehat{\Phi}\in \mathcal{M}'\cap \mathcal{M}_2$ the inverse Fourier transform of $y_{\Phi}$ and call it the Fourier multiple of $\Phi$.
Let $\Omega_1$ be the cyclic and separating vector in $L^2(\mathcal{M}_1)$ for $\mathcal{M}_1$.
Identifying $\widehat{\Phi}$ as an element in $\mathcal{M}'\cap \mathcal{M}_2$, we obtain that
    \begin{align}\label{eq:fouriermultiple finite  vN algebra}
        \widehat{\Phi}  x e_1 y \Omega_1 = \lambda^{-1/2}\sum_{j=1}^m x \eta_j^* e_1\Phi(\eta_j)y\Omega_1, \quad \text{ for all } x, y\in\mathcal{M}.
    \end{align}
    We have the algebraic expressions of bimodule maps as Proposition \ref{prop:phi formula} and Proposition \ref{prop:formula}.
\begin{proposition}
  Suppose $\mathcal{N}\subseteq\mathcal{M}$ is a finite inclusion of finite von Neumann algebras and $\Phi$ is a bounded bimodule map on $\mathcal{M}$. 
  Then we have 
  \begin{align*}
     \Phi(x)=\lambda^{-1} E_{\mathcal{M}}(y_{\Phi}xe_1),\quad x\in\mathcal{M}.
  \end{align*}
\end{proposition}

\begin{proposition}\label{prop:formula finite vN algebra}
Suppose $\mathcal{N}\subseteq \mathcal{M}$ is a finite inclusion of finite von Neumann algebras and  $\Phi$ is a bounded bimodule map on $\mathcal{M}$.  
Then 
\begin{align*}
    \Phi(x)=\lambda^{-3/2} E_{\mathcal{M}}(e_2e_1\widehat{\Phi} x e_1e_2).
\end{align*}
\end{proposition}
We also have the same result as Proposition \ref{prop:cpmapeq}.
\begin{proposition}\label{prop:cpmapeq finite vN algebra}
Suppose $\mathcal{N}\subseteq \mathcal{M}$ is a finite inclusion of finite von Neumann algebras and $\Phi\in \CB$.
Then $\Phi$ is completely positive if and only if $\widehat{\Phi} \geq 0$.
\end{proposition}
\begin{remark}
Suppose that $\mathcal{N}\subseteq\mathcal{M}$ is a finite inclusion of finite von Neumann algebras and $\Phi\in\CB$.
  In \cite{Zha23}, Zhao showed that when $\mathcal{N}\subseteq \mathcal{M}$ admits a downward basic construction, the positivity of $\Phi$ implies its complete positivity.
  By considering the finite inclusion $\mathbb{C}\subseteq M_{2}(\mathbb{C})$, we see that the transpose is positive but not completely positive.
\end{remark}
\begin{remark}
Suppose that $\mathcal{N}\subseteq\mathcal{M}$ is a finite inclusion of finite von Neumann algebras with the finite index $\lambda^{-1}$ and $\Phi\in\CP$.
    In \cite{Zha23}, Zhao showed that when $\mathcal{N}\subseteq \mathcal{M}$ admits a downward basic construction $\mathcal{N}_{-1}\subseteq \mathcal{N}\subseteq ^{e_{-1}}\mathcal{M}$, $\lambda_{\mathcal{N}\subseteq \mathcal{M}}=\lambda^{-1}\|\Phi(e_{-1})\|$, where $\lambda_{\mathcal{N}\subseteq \mathcal{M}}$ is the Pimsner-Popa constant.
    This implies that
    \begin{align*}
        \left\|(\Phi(\eta_k\eta_j^*))_{j,k=1}^m \right\|=\lambda^{-1} \|\Phi(e_{-1})\|.
    \end{align*}
\end{remark}

\begin{lemma}
    Suppose $\mathcal{N}\subseteq\mathcal{M}$ is a finite inclusion of finite von Neumann algebras and $\Phi\in\mathbf{CP}_{\mathcal{N}}(\mathcal{M})$ is contractive with $r(\Phi)=1$.
We assume there exists a faithful normal  state $\omega$ on $\mathcal{M}$ such that $\omega\Phi=\omega$.
Then for any $x\in\mathcal{M}(\Phi,\alpha)$, $|\alpha|=1$,
\begin{align*}
x\widehat{\Phi}^{1/2}e_1e_2=\alpha \widehat{\Phi}^{1/2}e_1e_2 x,\quad e_2e_1\widehat{\Phi}^{1/2}x=\alpha xe_2e_1\widehat{\Phi}^{1/2}. 
\end{align*}
\end{lemma}
\begin{proof}
    The proof is similar to the one of Lemma \ref{lem:commute}.
\end{proof}
We can obtain the following theorem by using the same proof as in Theorem \ref{thm:entry}.
\begin{theorem}\label{thm:entry finite vN algebra}
Suppose $\mathcal{N}\subseteq\mathcal{M}$ is a finite inclusion of finite von Neumann algebras and $\Phi$ is a bimodule quantum channel.
We assume that there exists a faithful normal  state $\omega$ on $\mathcal{M}$ such that $\omega\Phi=\omega$.
Then the following statements hold.
\begin{enumerate}[(i)]
\item  $\mathcal{M}(\Phi,\alpha)=\mathcal{P}(\Phi,\alpha)$ for any $\alpha\in\sigma(\Phi)\cap U(1)$.
\item $\mathcal{M}(\Phi,\alpha)^*=\mathcal{M}(\Phi,\overline{\alpha})$ for any $\alpha\in\sigma(\Phi)\cap U(1)$.
\item $\mathcal{M}(\Phi,\alpha_1)\mathcal{M}(\Phi,\alpha_2)\subseteq \mathcal{M}(\Phi,\alpha_1\alpha_2)$ for any $\alpha_1,\alpha_2\in\sigma(\Phi)\cap U(1)$.
In particular, $\mathcal{M}(\Phi,\alpha)\mathcal{M}(\Phi,\overline{\alpha})$ is a $*$- subalgebra of $\mathcal{M}(\Phi,1)$.

 We further suppose $\mathcal{M}(\Phi,1)$ is a factor.
\item The set $\Gamma:=\sigma(\Phi)\cap U(1)$ is a finite cyclic group.
\item
Then there exists a unitary $u_{\alpha}\in\mathcal{M}(\Phi,\alpha)$, $\alpha\in\Gamma$, such that 
\begin{align*}
\mathcal{M}(\Phi,\alpha)=u_{\alpha}\mathcal{M}(\Phi,1)=\mathcal{M}(\Phi,1)u_{\alpha}.
\end{align*}
Moreover, $\mathcal{M}(\Phi,\alpha)^{n}=\mathcal{M}(\Phi,\alpha^n)$, $n\in\mathbb{N}$.
In particular, $\mathcal{M}(\Phi,\alpha)^{|\Gamma|}=\mathcal{M}(\Phi,1)$, where $|\Gamma|$ is the order of the finite cyclic group $\Gamma$. 
\end{enumerate}
\end{theorem}
We give several characterizations of relative irreducibility as follows.
\begin{theorem}\label{thm: irreducible condition}
    Suppose $\mathcal{N}\subseteq\mathcal{M}$ is a finite inclusion of finite von Neumann algebras and $\Phi\in \CP$.
    Let $\displaystyle B=\bigvee_{k=0}^\infty \mathcal{R}\left(\widehat{\Phi}^{(*k)}\right)$.
Then the following statements satisfy (i) $\Rightarrow$ (ii) $\Leftrightarrow$ (iii) 
\begin{enumerate}[(i)]
\item $B=\mathbf{1}$;
    \item $\Phi$ is relatively irreducible;
    \item For any projection $p\in\mathcal{M}$,
    \begin{align*}
        \mathcal{R}((\Phi+id)^{d-1}(p))\in\mathcal{N},
    \end{align*}
    where $d=\dim_{\mathbb{C}}(\mathcal{M}'\cap\mathcal{M}_2)$.
\end{enumerate}
\end{theorem}
\begin{proof}
(i) $\Rightarrow$ (ii)
 Suppose $\Phi(p)\leq c p$ for some projection $p\in\mathcal{M}$ and constant $c>0$.
Since $\widehat{\Phi}\in\mathcal{M}'\cap\mathcal{M}_2$, which is finite dimensional, there exists $N>0$ such that $\displaystyle \mathbf{1}=\bigvee_{k=0}^N \mathcal{R}\left(\widehat{\Phi}^{(*k)}\right)$.
Therefore, there exists a constant $c_1>0$ such that $\displaystyle c_1\mathbf{1}\leq\sum_{k=0}^N \widehat{\Phi}^{(*k)}$.
By Remark \ref{rem:examples} and Remark \ref{rem:order relation}, there exists a constant $c_2>0$ such that $\displaystyle c_2E_{\mathcal{N}}\leq \sum_{k=0}^N\Phi^k$.
By the Pimsner-Popa inequality, we have
\begin{align*}
    \lambda_{\mathcal{N}\subseteq\mathcal{M}} p\leq E_{\mathcal{N}}(p)\leq c_2^{-1}\sum_{k=0}^N\Phi^k(p)\leq c_3p,
\end{align*}
for some constant $c_3>0$.
Hence $p=\mathcal{R}(E_{\mathcal{N}}(p))\in\mathcal{N}$, and then $\Phi$ is relatively irreducible.

(ii) $\Rightarrow$ (iii) 
Note that $\mathcal{R}((\widehat{\Phi}+e_1)^{\ast k})\leq\mathcal{R}((\widehat{\Phi}+e_1)^{\ast (k+1)}) $ for any $k\geq 0$.
Since $\mathcal{M}'\cap\mathcal{M}_2$ is finite dimensional, there exists $k_0\leq d-1$ such that $\mathcal{R}((\widehat{\Phi}+e_1)^{\ast k_0})=\mathcal{R}((\widehat{\Phi}+e_1)^{\ast( k_0+1)}) $.
By Remark \ref{rem:examples} and Remark \ref{rem:order relation}, there exist constants $c_1,c_2>0$ such that 
\begin{align*}
   c_1 (\Phi+id)^{k_0+1}\leq  (\Phi+id)^{k_0}\leq c_2 (\Phi+id)^{k_0+1}.
\end{align*}
Hence for any projection $p\in\mathcal{M}$, $\mathcal{R}((\Phi+id)^{k_0}(p))=\mathcal{R}((\Phi+id)^{k_0+1}(p))$.
By Kadison-Schwarz inequality, 
\begin{align*}
    \Phi(\mathcal{R}((\Phi+id)^{k_0}(p)))\leq \mathcal{R}(\Phi((\Phi+id)^{k_0}(p)) )\leq\mathcal{R}((\Phi+id)^{k_0+1}(p))=\mathcal{R}((\Phi+id)^{k_0}(p)).
\end{align*}
    Since $\Phi$ is relatively irreducible, $\mathcal{R}((\Phi+id)^{k_0}(p))\in\mathcal{N}$.
    
    (iii) $\Rightarrow$ (ii) Suppose $\Phi(p)\leq c p$ for some projection $p\in\mathcal{M}$ and constant $c>0$.
Then $p\leq (\Phi+id)^{d-1}(p)\leq c_1p$ for some $c_1>0$, and hence $p=\mathcal{R}((\Phi+id)^{d-1}(p))\in\mathcal{N}$.
Thus $\Phi$ is relatively irreducible.
\end{proof}
\begin{lemma}\label{lem:faithful state finite vN}
   Suppose $\mathcal{N}\subseteq\mathcal{M}$ is a finite inclusion of finite von Neumann algebra and $\Phi$ is a relatively irreducible bimodule quantum channel on $\mathcal{M}$.
   Suppose $\mathcal{N}$ is a factor.
   Then there exists a normal faithful state $\omega$ on $\mathcal{M}$ such that $\omega\Phi=\omega$.
\end{lemma}
\begin{proof}
    The proof is similar to the one of Lemma \ref{lem:faithful state}.
\end{proof}

We can also obtain the following theorem by using the same proof as in Theorem \ref{thm:RI Frobenius factor}.
We shall note that the condition that $\mathcal{N}$ is a factor is essential in the proof.
\begin{theorem}[Relative Irreducibility]\label{thm:RI Frobenius finite vN algebra}
    Suppose $\mathcal{N}\subseteq\mathcal{M}$ is a finite inclusion of finite von Neumann algebras and $\Phi$ is a relatively irreducible bimodule quantum channel.
    Suppose $\mathcal{N}$ is a factor.
  Then the eigenvalues of $\Phi$ with modulus $1$ form a finite cyclic subgroup $\Gamma$ of the unit circle $ U(1)$.
  The  space of fixed points $\mathcal{M}(\Phi,1)=\mathcal{N}$.
  For each $\alpha\in\Gamma$, there exists a unitary $u_{\alpha}\in\mathcal{M}(\Phi,\alpha)$ such that $\mathcal{M}(\Phi,\alpha)=u_{\alpha}\mathcal{N}=\mathcal{N}u_{\alpha}$.
\end{theorem}
 Theorem \ref{thm:RI Frobenius finite vN algebra} implies the quantum Frobenius theorem of Evans and H{\o}egh-Krohn \cite{EvaHoe78} as follows.
\begin{corollary}\label{cor:Frob}
    Suppose $\Phi:\mathcal{M}\to \mathcal{M}$ is an irreducible quantum channel and $\mathcal{M}$ is finite-dimensional.
    Then $\Gamma:=\sigma(\Phi) \cap  U(1)$ is a finite cyclic group and $\mathcal{M}(\Phi, \alpha)$ consists of a multiple of unitary elements.
\end{corollary}
\begin{proof}
It follows from Theorem \ref{thm:RI Frobenius finite vN algebra} when $\mathcal{N}=\mathbb{C}\mathbf{1}$ and $\mathcal{M}$ is finite dimensional.
\end{proof}

Evans and H{\o}egh-Krohn introduced the Collatz-Wielandt formula for positive maps on finite dimensional C$^*$ algebras.
The Collatz-Wielandt formula is not essential in our proofs.
However, we still have the Collatz-Wielandt formula for relatively irreducible bimodule quantum channels.
\begin{proposition}
Suppose $\mathcal{N}\subseteq\mathcal{M}$ is finite index and $\Phi\in\CP$ with the spectral radius $r=1$.
Suppose that one of the following two conditions holds: (1) $\Phi^*(\mathbf{1})=\mathbf{1}$; (2) $\Phi$ is a relatively irreducible bimodule quantum channel and $\mathcal{N}$ is a factor. 
Then the following statements hold:
\begin{enumerate}[(i)]
    \item If there exists a non-zero positive element $x \in\mathcal{M}$ such that $\Phi(x)\leq x$ or $\Phi(x)\geq x$, then $ \Phi(x)= x$;
    \item  If there exists a non-zero positive element $x \in\mathcal{M}$ such that $\Phi(x) = \widetilde{r} x$, then $\widetilde{r}=1$.
\end{enumerate}    
\end{proposition}
\begin{proof}
We first assume that $\Phi^*(\mathbf{1})=\mathbf{1}$.
 Suppose that $\Phi(x)\leq x$ for some nonzero positive element $x\in \mathcal{M}$.
Since $\Phi^*E_{\mathcal{N}}=\Phi^*(\mathbf{1})E_{\mathcal{N}}=E_{\mathcal{N}}$, we have $E_{\mathcal{N}}^*\Phi=E_{\mathcal{N}}^*$.
Hence $E_{\mathcal{N}}\Phi=E_{\mathcal{N}}$.
We have
\begin{align*}
E_{\mathcal{N}}(x-\Phi(x)) = E_{\mathcal{N}}(x)-E_{\mathcal{N}}\Phi(x)=0.
\end{align*}
By the Pimsner-Popa inequality,
\begin{align*}
   0= E_{\mathcal{N}}(x-\Phi(x)) \geq\lambda_{\mathcal{N}\subseteq\mathcal{M}} (x-\Phi(x))\geq0.
\end{align*}
Hence  $\Phi(x)= x$.
Similarly, we can prove the case when $\Phi(x)\geq x$.
(ii) It follows from (i).    

We next assume that $\Phi$ is relatively irreducible and $\mathcal{N}$ is a factor. 
By Lemma \ref{lem:faithful state finite vN}, there exists a normal faithful state $\omega$ on $\mathcal{M}$  such that $\omega\Phi=\omega$.
If $\Phi(x)\leq x$, then $\omega(x-\Phi(x))=0$ implies that $\Phi(x)=x$ since $\omega$ is faithful.
Similarly, we can prove the case when $\Phi(x)\geq x$.
We complete the proof.
\end{proof}

\section*{Acknowlednements}
We would like to thank Ce Wang for providing useful literature on quantum Markov processes.
L. Huang, Z. Liu, and J. Wu were supported by grants from Beĳing Institute of
Mathematical Sciences and Applications. 
C. Jiang was supported by Hebei Natural Science Foundation (No. A2023205045) and  by National Natural Science Foundation of China (Grant No. 12471120).
Z. Liu was supported by Beĳing Natural Science Foundation Key Program (Grant No. Z220002) and by Beĳing Natural Science Foundation
(Grant No. Z221100002722017). 
J. Wu was supported by National Natural Science Foundation of China (Grant Nos. 12371124 and 12031004).

\bibliographystyle{plain}

\end{document}